\documentclass[12pt]{amsart}

\usepackage{amsmath}
\usepackage{amssymb}
\usepackage{amsthm}
\usepackage{mathrsfs}
\usepackage{comment}
\usepackage{hyperref}
\usepackage{xcolor}
\usepackage{enumerate}
\usepackage{bm}
\usepackage{mathtools}
\usepackage[cal=euler]{mathalfa}
\usepackage[top=1in, bottom=1.25in, left=1.25in, right=1.25in]{geometry}
\usepackage{parskip}
\usepackage{tikz}
\usetikzlibrary{graphs,shapes,calc,decorations.pathmorphing,decorations.markings}

\colorlet{defaultgreen}{green!100}
\newcommand{\darkercolor}[3]{
    \colorlet{#3}{#1!#2!black}
}
\darkercolor{defaultgreen}{50}{darkgreen}

\newcommand{\mathds}[1]{#1}

\hypersetup{colorlinks=true,linkcolor=magenta,citecolor=blue}

\DeclareMathAlphabet{\mathpzc}{OT1}{pzc}{m}{it}

\theoremstyle{plain}
\newtheorem{theorem}{Theorem}[section]
\newtheorem*{theorem*}{Theorem}
\newtheorem{maintheorem}{Theorem}
\newtheorem{lemma}[theorem]{Lemma}
\newtheorem{proposition}[theorem]{Proposition}

\newtheorem{corollary}[theorem]{Corollary}
\newtheorem{conjecture}[theorem]{Conjecture}

\theoremstyle{definition}
\newtheorem{definition}[theorem]{Definition}

\newtheorem{question}[theorem]{Question}

\theoremstyle{remark}
\newtheorem{remark}[theorem]{Remark}

\newtheorem{egr}[theorem]{Example}
\newenvironment{example}{\begin{egr}}{\hfill\qedsymbol \end{egr}}

\numberwithin{equation}{section}
\numberwithin{figure}{section}

\newcommand{\val}{\mathrm{val}}
\newcommand{\R}{\mathfrak{o}} %integer ring
\newcommand{\PP}{\mathfrak{p}} %max ideal
\newcommand{\ratk}{F}  %rational field 
\newcommand{\resk}{\mathfrak{f}} %residue field
\newcommand{\Res}{\mathrm{Res}} %restriction
\newcommand{\Ind}{\mathrm{Ind}} %induction
\newcommand{\cind}{\textrm{c-}\mathrm{Ind}} %compact induction
\newcommand{\cInd}{\textrm{c-}\mathrm{Ind}}
\newcommand{\Hom}{\mathrm{Hom}}
\newcommand{\Spec}{\mathrm{Spec}}
\newcommand{\Rep}{\mathrm{Rep}}

\newcommand{\SL}{\mathrm{SL}}

\newcommand{\GL}{\mathrm{GL}}
\newcommand{\Gal}{\mathrm{Gal}}

\newcommand{\Z}{\mathbb{Z}} %integers
\newcommand{\bbR}{\mathbb{R}}

\newcommand{\bG}{\mathbf{G}}
\newcommand{\bS}{\mathbf{S}}
\newcommand{\bL}{\mathbf{L}}
\newcommand{\bZ}{\mathbf{Z}}

\newcommand{\bH}{\mathbf{H}}
\newcommand{\bT}{\mathbf{T}}
\newcommand{\bC}{\mathbf{C}}
\newcommand{\bU}{\mathbf{U}}
\newcommand{\bM}{\mathbf{M}}

\newcommand{\bounded}{\mathrm{b}}

\newcommand{\sG}{\mathsf{G}}
\newcommand{\sT}{\mathsf{T}}

\newcommand{\sH}{\mathsf{A}}
\newcommand{\sP}{\mathsf{P}}
\newcommand{\sM}{\mathsf{M}}
\newcommand{\sN}{\mathsf{N}}
\newcommand{\sQ}{\mathsf{Q}}
\newcommand{\sL}{\mathsf{L}}
\newcommand{\sU}{\mathsf{U}}

\newcommand{\apart}{\mathscr{A}}

\newcommand{\buil}{\mathscr{B}}

\newcommand{\facet}{\mathscr{F}}

\newcommand{\cJ}{\mathcal{J}}
\newcommand{\cN}{\mathcal{N}}
\newcommand{\cZ}{\mathcal{Z}}
\newcommand{\cH}{\mathcal{H}}
\newcommand{\cW}{\mathcal{W}}

\newcommand{\alg}{\mathrm{alg}}

\newcommand{\der}{\mathrm{der}}
\newcommand{\cosrepa}{a}
\newcommand{\cosrepb}{b}
\newcommand{\bz}{\mathpzc{z}}
\newcommand{\proj}{\mathrm{proj}}
\newcommand{\red}{\mathrm{red}}
\newcommand{\un}{\mathrm{un}}

\newcommand{\KimYugroup}{J^0_M}
\newcommand{\KimYusc}{J^0}

\begin{document}

\setlength{\parindent}{0pt}
\title{Typical representations via fixed point sets in Bruhat-Tits buildings}
\author{Peter Latham}
\address{Department of Mathematics and Statistics, University of Ottawa, Ottawa, Canada}
\email{platham@uottawa.ca}
\thanks{The first author's research was supported by the Heilbronn Institute for Mathematical Research.}

\author{Monica Nevins}
\address{Department of Mathematics and Statistics, University of Ottawa, Ottawa, Canada}
\email{mnevins@uottawa.ca}
\thanks{The second author's research is supported by a Discovery Grant from NSERC Canada.}
\date{\today}

\begin{abstract}
\noindent For an essentially tame supercuspidal representation $\pi$ of a connected reductive $p$-adic group $G$, we establish two distinct and complementary sufficient conditions for the irreducible components of its restriction to a maximal compact subgroup to occur in a representation of $G$ which is not inertially equivalent to $\pi$. These two results are further formulated in terms of the geometry of the Bruhat-Tits building of $G$ and its fixed points under the action of certain tori.  The consequence is a set of broadly applicable tools for addressing the branching rules of $\pi$ and the unicity of $[G,\pi]_G$-types. 
\end{abstract}
\maketitle
%%% \setlength{\parskip}{0pt}
%%% \tableofcontents
\setlength{\parskip}{12pt}
\setlength{\parindent}{0pt}

\section{Introduction}

One of the most fruitful tools for studying the (smooth, complex) representation theory of a reductive $p$-adic group $G=\bG(F)$ is the \emph{theory of types}. Given an irreducible representation $\pi$ of $G$, a type $(J,\lambda)$ for $\pi$ is an irreducible representation $\lambda$ of a compact open subgroup $J$ of $G$ such that containing $\lambda$ upon restriction to $J$ gives a necessary and sufficient condition for an irreducible representation of $G$ to be inertially equivalent to $\pi$, in the sense of \cite{BushnellKutzko1998}. Constructions of types are central to many recent developments in the representation theory of $p$-adic groups, and in particular to the explicit constructions of supercuspidal representations \cite{BushnellKutzko1993,Morris1999,Yu2001,Stevens2008,SecherreStevens2008}. Moreover, the theory of types has been shown to mirror much of the structure of the representation theory of $p$-adic groups in the more traditional sense: in particular, the Bushnell--Kutzko theory of covers gives an analogue of parabolic induction, and any instance of the local Langlands correspondence is expected to give rise to an inertial Langlands correspondence relating types to representations of the inertia group of $F$.

With types playing such a fundamental role in the representation theory of $p$-adic groups, it is natural to expect that they are somewhat hard to come by. Indeed the \emph{unicity of types} is the expectation that, for a supercuspidal representation $\pi$ of $G$ admitting a type $(J,\lambda)$, all other types for $\pi$ must arise from $(J,\lambda)$ by a series of minor representation-theoretic renormalizations (see Conjecture \ref{conj:unicity}). The unicity of types is now known in many special cases, specifically for split groups of type A \cite{Paskunas2005,Latham2016,Latham2018}, for depth-zero representations \cite{Latham2017}, and for many toral representations \cite{LathamNevins2018}; there are also some results towards the unicity of types for non-cuspidal representations \cite{Nadimpalli2017,Nadimpalli2019}. Our goal in this paper is to study the unicity of types for ``almost all'' supercuspidal representations. Our result explicitly makes use of the geometry of the Bruhat--Tits building $\buil(G)$ of $G$, and offers far more general results than have been obtained to date.

Specifically, we restrict attention to the supercuspidal representations and types constructed by J.K. Yu in \cite{Yu2001}, generalizing a previous construction due to Adler \cite{Adler1998}; following the terminology of \cite{BushnellHenniart2005}, we refer to such supercuspidal representations as \emph{essentially tame}. In almost all cases, restricting attention to these essentially tame representations is a vacuous condition: J. Fintzen proved in \cite{Fintzen2018}, extending earlier results of J. Kim \cite{Kim2007}, that every supercuspidal representation of $G$ is essentially tame if $p$ is coprime to the order of the Weyl group of $G$.

In order to describe the possible types contained within an essentially tame supercuspidal representation $\pi$ of $G$, we consider more generally the question of branching rules for $\pi$ upon restriction to a maximal compact subgroup $K$ of $G$. An equivalent formulation of the unicity of types is the assertion that any type contained in $\pi|_K$ (if there are any) must be induced from a type arising via J.K. Yu's construction. The determination of complete branching rules is a difficult problem which has only been solved in a few special cases, including $\GL_2(F)$  \cite{Casselman1973,Hansen1987}, $\SL_2(F)$ \cite{Nevins2005,Nevins2013}, unramified principal series of $\GL_3(F)$ \cite{CampbellNevins2010,OnnSingla2014}, and some partial results in the general depth zero case \cite{Nevins2014}. One of the impediments is that the description of the dual of $K$ remains an open problem.  Describing relationships between the branching rules for various families of representations of $G$ provides a valuable avenue to describing this dual.

We now give a brief description of our methods and results. Let $\pi$ be an irreducible, essentially tame supercuspidal representation of $G$, constructed from a datum $\Sigma$ as in \cite{Yu2001}. Let $(J,\lambda)=(J(\Sigma),\lambda_\Sigma)$ denote the type associated to $\Sigma$ in \cite{Yu2001}, and fix a maximal compact subgroup of $G$, which must coincide with the stabilizer $G_y$ of some point $y\in\buil(G)$. Since the restriction to $G_y$ of $\Pi:=\cInd_J^G\ \lambda$ is isomorphic to a direct sum of copies of $\pi|
_{G_y}$, the components we wish to consider are those of the Mackey decomposition of $\Pi$, which we show that we can rewrite as
\[\Pi|_{G_y}=\bigoplus_{g\in G_y\backslash G/J}{}^g\tau(y,g),
\]
where each \emph{Mackey component} $\tau(y,g)$ is a finite-dimensional representation of $G_{g^{-1}y}$ and $g^{-1}y$ ranges over the $G$-orbit of $y$ in $\mathscr{B}(G)$. We relate the capacity of $\tau(y,g)$ to contain a type to the position of the point $g^{-1}y$ in relation to the building-theoretic ingredients in the datum defining $\pi$: the building of a twisted Levi subgroup $G^0$, and a vertex $x$ of $\buil(G^0)$, viewed as a point in $\buil(G)$ via a generic embedding.

We prove the following two theorems which, during this introduction, are stated in a strictly weaker form for the sake of clarity.

Any point $z\in\buil(G)$ has a unique closest point in $\buil(G^0)$; this defines a projection map $\buil(G)\rightarrow\buil(G^0)$. Our first result describes a relationship between the structure of $\tau(y,g)$ and the image of the point $g^{-1}y\in\buil(G)$ under this projection map.

\begin{maintheorem}\label{thm:intro-1}[{see Sections~\ref{sec:non-cusp} and \ref{sec:projection}}]
Suppose that the projection of $g^{-1}y$ onto $\buil(G^0)$ lies in a facet distinct from $x$. Then every irreducible component $\tau$ of $\tau(y,g)$ is contained in an irreducible, essentially tame non-cuspidal representation of $G$, and hence is not a $[G,\pi]_G$-type.
\end{maintheorem}

More precisely, in Section~\ref{sec:non-cusp} we prove this result under a  
%In fact, we prove an analogous under a demonstrably 
weaker hypothesis about $J \cap G_{g^{-1}y}$.
% which requires a considerable amount of additional setup. 
In section \ref{sec:projection}, we relate this hypothesis to the projection map discussed above, and use this to both identify a large set of points in $\buil(G)$ to which Theorem \ref{thm:intro-1} applies, and a smaller subset to which it cannot.

For our second result, we identify a finite-index subgroup $H$ of $J$, equipped with a natural Moy--Prasad-like filtration by subgroups $H_t$, for $t\geq 0$. Each group $H_t$ fixes pointwise a compact neighbourhood $\Omega_t$ of $x\in\buil(G)$, and these neighbourhoods increase with $t$. Similarly, the filtration subgroups of the center $Z^0$ of $G^0$ fix pointwise an increasing family of $G^0$-invariant neighbourhoods $\Xi_t$ of $\buil(G^0)$.

\begin{maintheorem}\label{thm:intro-2}[{see Section~\ref{sec:seville}}]
  Suppose that the geodesic from $x$ to $g^{-1}y$ meets a point of $\Omega_{t+} \setminus \Xi_t$.  Then every irreducible component $\tau$ of  $\tau(y,g)$  must occur in the restriction of an irreducible representation which is not inertially equivalent to $\pi$.  Hence $\tau$ is not a $[G,\pi]_G$-type.
  \end{maintheorem}

We now describe the strategy of proof for each theorem. To prove Theorem \ref{thm:intro-1} we construct, for each irreducible representation $\tau$ occuring in $\tau(y,g)$, a non-cuspidal representation containing $\tau$, by carefully shifting from the point $x$ to a facet adjacent to $x$; to this facet we are able to associate a type for an essentially tame non-cuspidal representation following \cite{KimYu2017}, and show that this non-cuspidal representation must contain $\tau$. This argument significantly generalizes the central argument of \cite{Latham2017}, which applied only in the case that $\pi$ is depth-zero.

To prove Theorem \ref{thm:intro-2}, we perturb the \emph{simple character} of our datum in a manner which does not alter the Mackey component under consideration (but also does not necessarily produce the simple character of a new representation). We then show that the existence of such a perturbation violates the capacity of the Mackey component to lie in the restriction of a single inertial class of supercuspidal representations. This significantly generalizes an argument first presented in \cite{Paskunas2005}, and then carried out for all toral supercuspidal representations in \cite{LathamNevins2018}.

While these two results provide new, interesting, and explicit information regarding the branching rules of essentially tame supercuspidal representations, our major motivation in proving these results is to develop building-theoretic tools towards a proof of the unicity of types which work in the most general setting currently available. 

We show that the usual statement of the unicity of types may be reformulated in our notation (see Conjecture~\ref{conj:unicity-2}) as the statement that $\tau(y,g)$ contains a representation which is typical for $\pi$ if and only if $g^{-1}y\in\buil(G)$ is fixed by the action of $J$. It therefore remains to show that, for any point $y$ of $\buil(G)$ which is \emph{not} fixed by $J$, it must be the case that $g^{-1}y$ satisfies the hypotheses of either of the above two theorems. While the hypotheses of Theorem \ref{thm:intro-1} are easily understood, understanding their relation with those of Theorem \ref{thm:intro-2} appears to be intertwined with a number of explicit open questions regarding the geometry of $\buil(G)$, on which we elaborate in the final section. In particular, we state some explicit, open questions about the structure of $\buil(G)$ and the representation theory of the various reductive quotients of parahoric subgroups of $G$; our arguments show that the unicity of types follows from---and is essentially equivalent to---these questions.

The paper is organized as follows.  We briefly set some notation in Section~\ref{sec:notation} before establishing the necessary background in Bruhat--Tits theory in Section~\ref{sec:buildings}.  In Section~\ref{sec:types} we recap the general theory of types, before turning to the essential tame types which are the subject of this paper in Section~\ref{sec:kimyu}.  We lay out our strategy and define our Mackey components in Section~\ref{sec:mackey}.  Section~\ref{sec:non-cusp} is devoted to the statement and proof of Theorem~\ref{thm:non-cusp}.  In Section~\ref{sec:projection}, we describe the projection onto a sub-building and use these ideas to give two building-theoretic conditions under which Theorem~\ref{thm:non-cusp} can or cannot hold.  We also illustrate these with an example in $\mathbf{Sp}_4(F)$, which provides a lead-in to  Section~\ref{sec:seville}, which is devoted to the statement and proof of Theorem~\ref{thm:seville}. In  Section~\ref{sec:unicity} we explore the building-theoretic interpretation of the hypotheses of this theorem, and conclude with a discussion of the implications of our results.  In particular, we contrast our situation with those in which the unicity of types is known, and explain how our results appear to reduce the problem to an open question regarding finite groups of Lie type.

\subsection*{Acknowledgements}

%The authors thank Maarten Solleveld for discussions about the vagaries of fixed points of tori on buildings.  
The authors gratefully acknowledge the support of both the Centre International de Rencontres Math\'{e}matiques (Luminy) and the Mathematisches Forschungsinstitut Oberwolfach, who hosted them for two-week Research in Pairs stays in 2018 and 2019, respectively.  The warm hospitality of these institutes made this work possible.

\section{Notation} \label{sec:notation}

We now establish some basic notation which will be used freely throughout the paper. Let $\ratk$ be a field which is locally compact and complete relative to a normalized discrete valuation $\val$. Let $\R\subset\ratk$ be ring of integers, with maximal ideal $\PP\subset\R$ and residue field $\resk=\R/\PP$.
Let $\ratk^{\un}$ denote a maximal unramified closure of $\ratk$.

Let $\bG$ be a connected reductive algebraic group defined over $\ratk$, and write $G=\bG(\ratk)$ for its group of $\ratk$-rational points, equipped with its locally profinite topology. %Throughout the paper, we will be required to work with many different subgroups of $\bG$ and of $G$ in their respective topologies; 
We reserve the use of bold symbols for algebraic groups and group schemes defined over an $\R$-algebra, latin characters for closed subgroups of $G$, and  $\mathsf{serif}$ font for subgroups of the $\resk$-rational points of the special fibres of certain group schemes over $\R$.  We will often refer to a closed subgroup $H$ of $G$ of the form $\bH(F)$ for some closed subgroup $\bH$ of $\bG$ as having a property if $\bH$ has that property. In particular, this provides a notion of tori, parabolic subgroups and Levi subgroups of $G$ that we will frequently use.

On occasion, we will need to discuss disconnected algebraic groups, the connected component of which is reductive. In such cases, we extend all usual terminology in the obvious way: a torus of such a group is a subgroup which intersects with the connected component as a torus, and similarly for parabolic and Levi subgroups.

For a locally profinite group $H$, we denote by $\Rep(H)$ the category of smooth complex representations of $H$, i.e. the category of (possibly infinite-dimensional) complex vector spaces $V$ equipped with an action of $H$ such that the stabilizer of any vector $v\in V$ is an open subgroup of $H$. Without exception, when we say ``representation'' during the remainder of the paper, we will mean ``smooth complex representation''.

We write $\Ind_H^G$ and $\cInd_H^G$ for the induction and compact induction functors $\Rep(H)\rightarrow\Rep(G)$, respectively. While we will usually omit restriction functors from the notation, we will occasionally denote them by $\Res_H^G:\Rep(G)\rightarrow\Rep(H)$.

Given $H\subseteq G$ and $g\in G$, set ${}^gH = \{ghg^{-1}\mid h\in H\}$ and for any $\rho \in \Rep(H)$, the corresponding representation ${}^g\rho \in \Rep({}^gH)$ is given on $k\in {}^gH$ by ${}^g\rho(k) = \rho(g^{-1}kg)$.

\section{Bruhat--Tits theory}\label{sec:buildings}

First, note that for any algebraic torus $\bT$ defined over $F$, there exists an \emph{lft-N\'{e}ron model} of $\bT$, as defined in \cite{BoschLutkebohmertRaynaud1990}. We denote this lft-N\'{e}ron model by $\bT_\bounded$. This is a smooth affine $\R$-group scheme which is locally of finite type and has generic fibre $\bT_\bounded\times_{\Spec\ \R}\Spec\ F=\bT$. %, and which satisfies the N\'{e}ron mapping property.
Write $\bT_0$ for its connected component.  Then $T_\bounded:=\bT_\bounded(\R)$ is the maximal bounded subgroup of $T:=\bT(F)$, and $T_0:=\bT(\R)$ is a finite index subgroup of $T_\bounded$ called the \emph{parahoric subgroup} of $T$. The \emph{Moy--Prasad filtration} of $T$ is the decreasing filtration $\{T_r\ | \ r\geq 0\}$ of $T_0$ by open subgroups defined by
\[T_r=\{t\in T_0\ | \ \val(\chi(t)-1)\geq r\text{ for all }\chi\in X^*(\bT)\}.
\]

Now, and for the remainder of this section, suppose that $\bG$ is a connected reductive group defined over $F$ with centre $Z(\bG)=\bZ{_\bG}$. Choose a maximal $F$-split torus $\bS$, contained in a maximal $F^{\un}$-split torus $\bS^{\un}$.  Since $\bG$ is quasi-split over $F^{\un}$, the centralizer $\bC$ of $\bS^{\un}$ is a (maximal) torus defined over $F$; in any case it is a minimal Levi subgroup of $\bG$.

Let $\Phi = \Phi(\bG,\bS,F)$ be the roots of $\bS$ in $\bG$ defined over $F$, and let $\Psi=\Psi(\bG,\bS,F)$ be the associated system of affine roots. As described carefully in \cite[\S2.2]{Fintzen2015}, the root subgroup $U_\alpha\subseteq G$, for $\alpha \in \Phi$, admits a filtration by compact open subgroups $U_\psi$ indexed by those $\psi \in \Psi$ with gradient $\alpha$.

Let $X_*(\bS)$ denote the group of cocharacters of $\bS$.  The affine space $\apart = \apart(\bG,\bS,F) = X_*(\bS)\otimes_\Z \bbR$, called the apartment defined by $\bS$, carries a hyperplane structure defined by $\Psi$.  To each point $x\in \apart$, F.~Bruhat and J.~Tits \cite{BruhatTits1984} associated a \emph{parahoric subgroup}, which is a smooth affine $\R$-group scheme $\bG_{x,0}$ with the following properties:
\begin{enumerate}[(i)]
\item the generic fibre $\bG_{x,0}\times_{\Spec\ \R}\Spec\ \ratk$ of $\bG_{x,0}$ is equal to $\bG$;
\item the special fibre $\bG_{x,0}\times_{\Spec\ \R}\Spec\ \resk$ of $\bG_{x,0}$ is a connected reductive algebraic group over $\resk$;
\item  the group $G_{x,0}:=\bG_{x,0}(\R)$ of $\R$-points of $\bG_{x,0}$ is compact and open in $G$, and is given by
$$G_{x,0}=\langle C_0,U_\psi\ |\ \psi\in\Psi\ :\ \psi(x)\geq 0\rangle.
$$
where $C_0$ is defined as above if $\bC$ is a torus, and by Galois descent from $\bC_{x,0}(E)$ for a splitting field $E$ of $\bC$ otherwise (see Remark~\ref{rem:not-quasi-split}); the group $C_0$ is independent of the choice of $x$ in either case.
\end{enumerate}

The (enlarged or extended) \emph{Bruhat-Tits building} of $G$ is obtained by gluing together the apartments defined by all maximal $F$-split tori of $\bG$, as follows.  Set $\buil(G)=\buil(\bG,F):=(G\times\apart(\bG,\bS,F))/\sim$, where $(g,x)\sim(g',x')$ if and only if there exists an $n\in N_G(\bS(F))$ such that $nx=x'$ and $g^{-1}g'n\in G_{x,0}$. Then we identify $\{(g,x)\mid x\in \apart\}$ with the apartment corresponding to ${}^g\bS$.

Then $\buil(G)$ is a non-positively curved geodesic (CAT(0)) metric space.  The left-regular action of $G$ on $G\times\apart(\bG,\bS,F)$ descends to an action of $G$ on $\buil(G)$ via isometries.  Given any two points $x,y\in \buil(G)$, there exists an apartment $\apart$ containing both.  The geodesic $[x,y]$ is then a line segment in this apartment.

\begin{proposition}[{Bruhat--Tits \cite{BruhatTits1984}}]\label{prop:Gx-def}
For each point $x\in\buil(G)$, there exists a smooth affine $\R$-group scheme $\bG_x$ with finite component group such that:
\begin{enumerate}[(i)]
\item the group $G_x:=\bG_x(\R)$ of $\R$-points of $\bG_x$ is compact and equal to $\mathrm{Stab}_G(x)$;
\item for any apartment $\apart$ containing $x$, the connected component of $\bG_x$ coincides with the parahoric group scheme $\bG_{x,0}$ defined relative to $\apart$.
\end{enumerate}
\end{proposition}

\begin{remark}
We will often need to take great care in distinguishing between the point stabilizer subgroups $G_x$ and the parahoric subgroups $G_{x,0}$. One convenient method of describing the difference between these two groups is to use the \emph{Kottwitz homomorphism}, as defined in \cite{Kottwitz1997}. This is a homomorphism $\kappa$ from $G$ to the \emph{algebraic fundamental group} $\pi_1^\alg(G)$ of $G$, as in \cite{Borovoi1998}, which has the property that, for any $x\in\buil(G)$, one has $G_{x}\cap\ker\kappa = G_{x,0}$ \cite[Appendix]{PappasRappoport2008}. In fact, the Kottwitz homomorphism can be used to describe elements of $G_x\backslash G_{x,0}$ rather more generally: given a compact element $g\in G$, there exists a point $x\in\buil(G)$ such that $g\in G_x\backslash G_{x,0}$ if and only if $\kappa(g)$ is a non-trivial torsion element of $\pi_1^\alg(G)$.
\end{remark}

For any $x\in \buil(G)$, choose an apartment $\apart(\bG,\bS,F)$ containing $x$. A.~Moy and G.~Prasad defined a filtration $\{G_{x,r}\ | \ r\geq 0\}$ of $G_{x,0}$ by open normal subgroups by setting
\[G_{x,r}=\langle C_r,U_\psi\ | \ \psi\in\Psi(\bG,\bS,\ratk)\ : \ \psi(x)\geq r\rangle,
\]
where again, in the non-quasi-split case, the filtration subgroup $C_r$ of $C=\bC(F)$ is defined by Galois descent from an appropriate splitting field.

One final convenient observation is that the Moy--Prasad filtration groups $G_{x,r}$ are actually \emph{schematic}, which is to say that there are natural smooth affine $\R$-group schemes $\bG_{x,r}$ with generic fibre $\bG$ such that $\bG_{x,r}(\R)=G_{x,r}$ \cite{Yu2015}.  For example, if $x\in \apart$ we have $\bU_{\alpha,x,r}(\R) = \bigcup_\psi U_\psi$, where this union is over all affine roots $\psi$ of $\apart$ with gradient $\alpha$ such that $\psi(x)\geq r$.

Write $G_{x,r+}=\bigcup_{s>r}G_{x,s}$ and $G_{x,r:r+}=G_{x,r}/G_{x,r+}$. Then $G_{x,0:0+}$ coincides with the group of $\resk$-rational points of the special fibre of $\bG_{x,0}$, and is therefore a finite group of Lie type, and $G_{x,r:r+}$ is an abelian $p$-group for $r>0$.

\begin{remark}\label{rem:not-quasi-split}
When $\bG$ is not $F$-quasi-split, the $\bG$-centralizer $\bC$ of a maximal $F$-split torus $\bS\subset\bG$ need not be a torus. Since there exists a finite unramified Galois extension $E/F$ such that $\bG$ is $E$-quasi-split, we define the building by Galois descent from $E$. Specifically, one may choose a maximal $E$-split torus $\bT$ containing $\bS$ as its maximal $F$-split subtorus.  The apartment $\apart(\bG,\bT,E)$ then has an action of $\Gal(E/F)$ such that $\apart(\bG,\bT,E)^{\Gal(E/F)}=\apart(\bG,\bS,F)$. Together with the action of $\Gal(E/F)$ on $\bG(E)$, this allows one to define an action of $\Gal(E/F)$ on $\buil(\bG(E))$, and we define $\buil(G):=\buil(\bG(E))^{\Gal(E/F)}$. With this, we are able to define group schemes $\bG_{x,0}$, $\bG_x$ generalizing those of Proposition \ref{prop:Gx-def}, as well as $\bG_{x,r}$ as in \cite{Yu2015}.
\end{remark}

\subsection{The reduced building and other properties}

%We will occasionally prefer to work with a subset of $\buil(G)$ called the
The \emph{reduced building} of $G$ is $\buil^\red(G)=\buil(G_\der)$, where $G_\der=\bG_\der(F)$, for $\bG_\der\subset\bG$ the derived subgroup. One has $\buil(G)\cong \buil^\red(G)\times X_*(\bZ_\bG)\otimes_{\mathbb{Z}}\mathbb{R}$, and so there is a projection map $\buil(G)\rightarrow\buil^\red(G)$, which we denote by $x\mapsto[x]$. Later it will be convenient to abbreviate by $\tilde{[x]}$ the fibre of this projection over $[x]$.

An advantage of working with $\buil^{\red}(G)$ is that it is a polysimplicial complex; for example, we refer to a point of $\buil(G)$ as a vertex if its image $[x]$ is one.  We also have that $\bG_x=\bG_y$ if and only if $[x]=[y]$.  However, given $x\in\buil(G)$, the stabilizer $G_{[x]}$ of $[x]\in\buil^\red(G)$ is only a compact-modulo-centre subgroup of $G$, the maximal compact subgroup of which coincides with $G_x$; in fact $G_{[x]}=N_G(G_x)$.

By Proposition~\ref{prop:Gx-def}, $G_x$ is the stabilizer of $x \in \buil(G)$ under the action of $G$.  More generally, for a subset $\Omega \subseteq \buil(G)$ we use the notation
$$
G_\Omega = \{ g\in G \mid g x=x \ \forall x\in \Omega\}
$$
for the pointwise stabilizer of $\Omega$. 
Given a subgroup $H$ of $G$, we write $\buil(G)^{H}$ for its set of fixed points, which we can equivalently write as $\{x\in \buil(G) \mid H\subseteq G_x\}$.  

Since $\buil(G)$ is a CAT(0) space, if $g\in G$ fixes both $x$ and $y$, then it fixes the geodesic $[x,y]$.  This has several consequences.  For one, if $z\in [x,y]$ then $G_x \cap G_y \subseteq G_z$, a fact we use frequently in the sequel.  For another, it follows that for any $H$, $\buil(G)^{H}$ is convex; its image in $\buil^{\red}(G)$ is bounded if, for example, $H$ is compact open. 

%Given two points $x, z \in \buil(G)$ we consider, in section \ref{sec:non-cusp}, the set
%$$
%\Gamma(x,z) = \buil(G)^{G_x\cap G_z},
%$$
%which is the largest subset of $\buil(G)$ for which $G_x \cap G_z = G_{\Gamma(x,z)}$.  The geodesic $[x,z]$ is contained in $\Gamma(x,z)$, but this set is in general larger.  For one, it contains the fibre over $\buil^{\red}(G)$ of each of its points.  For another, when $\bG$ is simply connected, $\Gamma(x,z)$ is the simplicial closure of $[x,z]$; more generally it is convex closed subset thereof.
%
%
%

\section{The theory of types}\label{sec:types}

In this section, we recall the theory of types in the abstract sense laid out in \cite{BushnellKutzko1998}; the objective of this paper is to describe to what extent all such types, in the tame case, may be described in terms of the types constructed in \cite{Yu2001}.

\subsection{Bernstein decomposition}

A \emph{cuspidal pair} in $G$ is a pair $(M,\rho)$ consisting of a Levi subgroup $M$ of $G$ and an irreducible supercuspidal representation $\rho$ of $M$. Given $\pi$ a smooth irreducible representation of $G$, Jacquet's theorem implies there exists a unique $G$-conjugacy class of cuspidal pairs $(M,\rho)$ such that $\pi$ is isomorphic to a subquotient of $\Ind_P^G\ \rho$ for some parabolic subgroup $P$ of $G$ with Levi factor $M$. We call this conjugacy class the \emph{cuspidal support} of $\pi$.

We say that cuspidal pairs $(M,\rho)$ and $(M',\rho')$ are \emph{$G$-inertially equivalent} if there exists an unramified character $\omega$ of $M'$ and a $g\in G$ such that $^gM=M'$ and $^g\rho\simeq\rho'\otimes\omega$. We write $\frak{s}=[M,\rho]_G$ for the $G$-inertial equivalence class of $(M,\rho)$. The \emph{inertial support} of an irreducible representation of $G$ is the inertial equivalence class of its cuspidal support. We say that two irreducible representations are \emph{inertially equivalent} if they have the same inertial support, and write $\frak{I}(\pi)$ for the inertial equivalence class of $\pi$.

Write $\frak{B}(G)$ for the set of $G$-inertial equivalence classes of cuspidal pairs. Given any subset $\frak{S}$ of $\frak{B}(G)$, denote by $\Rep^{\frak{S}}(G)$ the full subcategory of $\Rep(G)$ consisting of those representations every irreducible subquotient of which has inertial support contained in $\frak{S}$.  By a theorem of Bernstein \cite{Bernstein1984}, we have a categorical decomposition 
\[\Rep(G)=\prod_{\frak{s}\in\frak{B}(G)}\Rep^{\frak{s}}(G).
\]

\subsection{Types and covers}

\begin{definition}
Let $\frak{S}\subset\frak{B}(G)$ be a finite set, and let $(J,\lambda)$ be a pair consisting of an irreducible representation $\lambda$ of a compact open subgroup $J$ of $G$.
\begin{enumerate}[(i)]
\item We say that $(J,\lambda)$ is $\frak{S}$\emph{-typical} if, for any irreducible representation $\pi$ of $G$ such that $\Hom_J(\lambda,\pi)\neq 0$, one must have $\pi\in\Rep^{\frak{S}}(G)$.
\item We say that $(J,\lambda)$ is an \emph{$\frak{S}$-type} if it is $\frak{S}$-typical and every irreducible representation $\pi$ in $\Rep^{\frak{S}}(G)$ satisfies $\Hom_J(\lambda,\pi)\neq 0$.
\end{enumerate}
\end{definition}

As simple applications of Frobenius reciprocity and the transitivity of compact induction, one immediately deduces a few properties of types:
\begin{enumerate}[(i)]
\item For any $g\in G$, the pair $(^gJ,^g\lambda)$ is an $\frak{S}$-type if and only if $(J,\lambda)$ is an $\frak{S}$-type.
\item If $(J,\lambda)$ is an $\frak{S}$-type, $K\supset J$ is a compact open subgroup of $G$ and $\tau$ is an irreducible representation of $K$ satisfying $\Hom_J(\lambda,\tau)\neq 0$, then $(K,\tau)$ is $\frak{S}$-typical.
\item If $\frak{s}=[G,\pi]_G$ is an inertial equivalence class of supercuspidal representations of $G$, then $(J,\lambda)$ is an $\frak{s}$-type if and only if it is $\frak{s}$-typical.
\end{enumerate}
In particular, combining (ii) and (iii) shows that any irreducible representation of a \emph{maximal} compact subgroup of $G$ which contains some $[G,\pi]_G$-type $(J,\lambda)$ must itself be a $[G,\pi]_G$-type.

Suppose now $M$ a Levi subgroup of $G$, and $(J_M,\lambda_M)$ an $\frak{S}_M$-type, where $\frak{S}_M=\{[L_i,\rho_i]_M\}\subset\frak{B}(M)$ is a finite set.  In \cite{BushnellKutzko1998}, C.~Bushnell and P.~Kutzko define the notion of a \emph{cover} $(J,\lambda)$ of $(J_M,\lambda_M)$ which, if it exists, is an $\frak{S}$-type, where $\frak{S}=\{[L_i,\rho_i]_G\}$.  In this case, we have the following analogue to the above lemma.

\begin{lemma}\label{lem:induced-from-type-nsc}
Suppose that $(J_M,\lambda_M)$ is an $[M,\rho]_M$-type and $(J,\lambda)$ is a $G$-cover of $(J_M,\lambda_M)$. Then $\cInd_J^G\ \lambda$ has a filtration by subrepresentations such that each successive quotient is a finite-length representation isomorphic to $\cInd_P^G\ \rho'$ for some $\rho' \in \frak{I}(\rho)$ and some parabolic subgroup $P$ of $G$ with Levi factor $M$.
\end{lemma}

\begin{proof}
Fix a parabolic subgroup $P=MN$ of $G$ with Levi factor $M$. By \cite[Th\'{e}or\`{e}me 2]{Blondel2005}, we have $\cInd_J^G\ \lambda=\cInd_{J_MN}^G\ \lambda_M\otimes\mathds{1}_N$, which yields
\[\cInd_J^G\ \lambda=\Ind_{MN}^G\ \left( \cind_{J_M}^M \lambda_M\right)\otimes\mathds{1}_N.
\]
The claim then follows from \cite[Proposition 5.2]{BushnellKutzko1998}.
\end{proof}

\section{Kim-Yu types}\label{sec:kimyu}
%\subsection{Twisted Levi sequences}

In this section, we recall the construction of types for essentially tame representations by J.K.~Yu and J.-L.~Kim as given in \cite{Yu2001} and \cite{KimYu2017}.  We largely follow the treatment given in \cite{HakimMurnaghan2008}, but we depart slightly from their notational conventions in some cases, so as to use $J$ for the compact open subgroups supporting types, as in \cite{BushnellKutzko1998}.  
For the reader familiar with these constructions the main substitutions are (a) $J$ in place of $K$ as the root letter in notation to denote an open compact-mod-centre subgroup,  (b) $\cJ$ in place of $J$ as the root letter for the groups like $(G^i,G^{i+1})_{r_i,s_i}$, (c) $H_+$ in place of $K_+$ as the largest subgroup on which the type acts via a character. % for a character $\theta$.  %the the where $J$ (respectively, $K_+$) was used in \emph{loc. cit.}, we have replaced it with $\cJ$ (respectively, $H_+$).

\subsection{Data}

\begin{definition}
A \emph{twisted Levi sequence} in $\bG$ is a sequence $\vec{\bG}=(\bG^0,\dots,\bG^d)$ of closed $F$-subgroups $\bG^i$ of $\bG$ such that $\bG^0 \subsetneq \bG^1\subsetneq \cdots \subsetneq \bG^{d}=\bG$ and such that there exists a finite algebraic extension $E/F$ such that each group $\bG^i(E)$ is a Levi subgroup of $\bG(E)$.
We say that $\vec{\bG}$ \emph{splits over $E$}.   A \emph{generalized twisted Levi sequence} is defined similarly but without the condition that adjacent elements be distinct.
\end{definition}

We require the definition of a (not necessarily cuspidal) datum given in \cite[7.2]{KimYu2017}. Due to their complexity, (\ref{item:M^i}), (\ref{item:gen-char}) and (\ref{item:gen-emb}) are not fully defined below.   Instead, the generic embeddings of (\ref{item:gen-emb}) will be discussed in detail where they are used in the proof of Lemma~\ref{lem:define-point-u-xi}; the genericity of the characters in (\ref{item:gen-char}) does not come into question in our construction, so we do not review it; but we will recap the construction of the groups $\bM^i$ of (\ref{item:M^i}) in Remark~\ref{rem:gen-twisted-levi}.

\begin{definition}\label{def:datum}
A \emph{datum} in $G$ is a tuple $\Sigma=((\vec{\bG},\bM^0),(x,\{\iota\}),(\KimYugroup,\sigma),\vec{r},\vec{\phi})$ consisting of
\begin{enumerate}[(i)]
\item a twisted Levi sequence $\vec{\bG}$ in $\bG$ which splits over a tamely ramified extension $E/F$;
\item \label{item:M^i} a Levi subgroup $\bM^0$ of $\bG^0$, which defines $F$-Levi subgroups $\bM^i$ of $\bG^i$ such that $\vec{\bM}$ is a generalized twisted Levi sequence of $\bM=\bM^d$ as in \cite[2.4]{KimYu2017}; %(see Remark~\ref{rem:gen-twisted-levi}, below);
\item a point $x\in\buil(M^0)$ for which $M_{x,0}^0$ is a maximal parahoric subgroup of $M^0=\bM^0(F)$;
\item a compact open subgroup $\KimYugroup$ of $M^0_x$ containing $M^0_{x,0}$ as a normal subgroup;
\item an irreducible representation $\sigma$ of $\KimYugroup$ such that $\sigma|_{M^0_{x,0}}$ is a sum of cuspidal representations of $M_{x,0:0+}^0$;
\item a sequence $\vec{r}=(r_0,\dots,r_d)$ of positive real numbers satisfying $r_0<r_1<r_2<\cdots<r_{d-1}\leq r_d$ if $d>0$ and $0\leq r_0$ if $d=0$;
\item \label{item:gen-char} a sequence $\vec{\phi}=(\phi^0,\dots,\phi^d)$ of (quasi-)characters $\phi^i$ of $G^i$, such that for $\phi^i$ is $G^{i+1}$-generic of depth $r_i$ in the sense of \cite[Def. 3.7]{HakimMurnaghan2008} and $\phi^d$ is trivial if $r_d=r_{d-1}$; and
\item  \label{item:gen-emb} an $\vec{s}$-generic (relative to $x$) diagram of embeddings $\{\iota\}$  of the buildings $\buil(M^i)$ and $\buil(G^i)$ (following the given inclusions of groups) in the sense of \cite[3.5]{KimYu2017}, where $\vec{s}=(s_{-1},s_0,s_1,\ldots, s_{d-1})$ with $s_{-1}=0$ and $s_i=r_i/2$ for $0\leq i< d$.
\end{enumerate}
In the case that $M^0=G^0$ and $Z(\bG^0)/\bZ_{\bG}$ is anisotropic, we say that $\Sigma$ is a \emph{cuspidal datum}; this is the case considered in \cite{Yu2001}. For a cuspidal datum, the choice of embeddinggs $\{\iota\}$ turns out to be unimportant, so we instead write briefly $\Sigma=(\vec{\bG},x,\KimYusc,\sigma,\vec{r},\vec{\phi})$, omitting the group $\KimYusc$ when $\KimYusc=G^0_x$ is maximal.
\end{definition}

\begin{remark}\label{rem:gen-twisted-levi}
The groups $\bM^i$ appearing in (ii) above are defined in \cite[\S 2.4]{KimYu2017} as follows.  Let $Z_\mathrm{s}(\bM^0)^\circ$ 
be the maximal $F$-split torus in the centre of $\bM^0$. For each $0\leq i\leq d$, let $\bM^i$ be the $\bG^i$-centralizer of $Z_\mathrm{s}(\bM^0)^\circ$. This defines a sequence $\vec{\bM}=(\bM^0,\bM^1,\ldots,\bM^d)$ that, as shown in \cite[\S 2.4]{KimYu2017}, is a generalized twisted Levi sequence in $\bM=\bM^d$, defined over $F$ and split over $E$.  We set $M^i=\bM^i(F)$.
\end{remark}

\subsection{Groups associated to the datum}
%To define the types Kim and Yu associated to a datum, we first need to define several subgroups used in the construction.
For the remainder of the section, fix a datum $\Sigma=((\vec{\bG},\bM^0), (x,\{\iota\}), (\KimYugroup,\sigma),\vec{r},\vec{\phi})$. Set $J_G^0=J_M^0G_{x,0}^0$; this is a subgroup of $G_x^0$ such that $J_G^0/G^0_{x,0+}=J_M^0/M^0_{x,0+}$.

For each $0\leq i\leq d$, define compact open subgroups of $G^i$ by
\begin{align*}
J^i(\Sigma,G)&=J_G^0G_{x,s_0}^1\cdots G_{x,s_{i-1}}^i; \text{ and}\\
%J^i_0(\Sigma,G)&=J_G^0G_{x,s_0}^1\cdots G_{x,s_{i-1}}^i;
H^i(\Sigma,G)&=J_G^0 G_{x,s_0+}^1\cdots G_{x,s_{i-1}+}^i.
\end{align*}
These are in fact the $\R$-points of the analogous affine normal $\R$-group schemes, each of whose generic fibres are equal to $\bG^i$. Where the pair $(\Sigma,G)$ is understood, we may choose to omit it from the notation.  

The induced Moy--Prasad filtration on each of these subgroups is defined, for each $t\geq 0$ and any $0\leq i \leq d$, as
$$
J^i_t:=J^i\cap G_{x,t},\text{ and}\ 
H^i_t:=H^i\cap G_{x,t}.
$$
We abbreviate $J^i_+:=J^i_{0+}$ and $H^i_+=H^i_{0+}$. Then we have a tower of compact open subgroups of $G^i$ given by
$$
H^i_+ \subseteq J^i_+ \subseteq J^i_0 \subseteq J^i
$$
such that $J^i_0/J^i_{+} \cong G^0_{x,0:0+}$ for every $0\leq i\leq d$.  By \cite[Proposition 4.3(b)]{KimYu2017}, we also have
\begin{equation}\label{eq:lift}
J^i/J^i_+ \cong \KimYugroup/M^0_{x,0+}.
\end{equation}

For each $0\leq i<d-1$, our next groups are defined first as subgroups of $\bG^{i+1}(E)$.  Let $\bT$ be a maximal $E$-split $F$-torus of $\bG^0$ and for each $0\leq j\leq d$ set $\Phi^j = \Phi(\bG^j,\bT,E)$.  Then we may define
\begin{align*}
\cJ^{i+1}(\Sigma,G,\R_E):=&\langle\bT_{r_i}(\R_E),\bU_{\alpha,x,r_i}(\R_E),\bU_{\beta,x,s_i}(\R_E)\ | \ \alpha\in\Phi^i,\ \beta\in\Phi^{i+1}\setminus \Phi^i\rangle;\text{ and}\\
\cJ^{i+1}_+(\Sigma,G,\R_E):=&\langle\bT_{r_i}(\R_E),\bU_{\alpha,x,r_i}(\R_E),\bU_{\beta,x,s_i+}(\R_E)\ | \ \alpha\in\Phi^i,\ \beta\in\Phi^{i+1}\setminus\Phi^i\rangle.
\end{align*}
We then define two more compact open subgroups of $G^{i+1}$ 
\begin{equation}\label{eq:def-script-J}
\cJ^{i+1}(\Sigma,G):=\cJ^{i+1}(\Sigma,G,\R_E)\cap G,\text{ and}\ 
\cJ^{i+1}_+(\Sigma,G):=\cJ^{i+1}_+(\Sigma,G,\R_E)\cap G.
\end{equation}
In the literature, the preceding is often abbreviated by writing
$$
\cJ^{i+1}(\Sigma,G) = (G^i,G^{i+1})_{r_i,s_i} \text{ and}\ \cJ^{i+1}_+(\Sigma,G)=(G^i,G^{i+1})_{r_i,s_i+},
$$
which we will use in the sequel for convenience.  Again, we may omit $(\Sigma,G)$ where this is clear from context.

Note that $\cJ^{i+1}$ can be thought of as a kind of complement to $J^i$ in $J^{i+1}$ in the sense that 
\begin{equation}\label{eq:product}
J^i \cJ^{i+1}=J^{i+1};
\end{equation}
but crucially, $J^i\cap \cJ^{i+1} = G^i_{x,r_i}$ is non-trivial.  These groups give us an alternate description of $H$ \cite[\S 3]{HakimMurnaghan2008} as
\begin{equation}\label{eq:product2}
H = J_G^0\cJ^1_+\cdots \cJ^d_+.
\end{equation}

Finally, from the datum $\Sigma$ (and Remark~\ref{rem:gen-twisted-levi}) we can also extract the tuple $\Sigma_M:=(\vec{\bM},x,J_M^0,\sigma,\vec{r},\vec{\phi})$.  The following lemma summarizes results proven in \cite[7.4]{KimYu2017}, using \cite{Yu2001}.

\begin{lemma}
The tuple $\Sigma_M$ is a cuspidal datum in $M$ and, for $0\leq i\leq d$, one has equalities
\begin{align*}
J^i(\Sigma,M)&:=J^i(\Sigma,G)\cap M^i=J^i(\Sigma_M,M) \text{ and}\\
%J^i_0(\Sigma,M)&:=J^i_0(\Sigma,G)\cap M^i=J^i_0(\Sigma_M,M), 
%J^i_+(\Sigma,M)&:=J^i_+(\Sigma,G)\cap M^i=J^i_+(\Sigma_M,M),
H^i(\Sigma,M)&:=H^i(\Sigma,G)\cap M^i=H^i(\Sigma_M,M).
\end{align*}
\end{lemma}

%\begin{lemma}\label{lem:Ji-to-Ji+1}
%For each $0\leq i\leq d-1$, we have $J^i(\Sigma)\cJ^{i+1}(\Sigma)=J^{i+1}(\Sigma)$.
%\end{lemma}

\subsection{The types $(J(\Sigma,G),\lambda_\Sigma)$ and $(J(\Sigma,M),\lambda_\Sigma)$}

Continuing with the notation as above, the restriction of the character $\phi^i$ to $G^i_{x,r_{i}}$ admits an extension $\hat{\phi}^{i}$ to $\cJ^{i+1}_+(\Sigma,G)$. We write $\cN^{i+1}(\Sigma,G)=\ker\hat{\phi}^{i}$.  Set 
\begin{align*}
\cH^{i+1}(\Sigma,G)&:=\cJ^{i+1}(\Sigma,G)/\cN^{i+1}(\Sigma,G),\\
\cW^{i+1}(\Sigma,G)&:=\cJ^{i+1}(\Sigma,G)/\cJ^{i+1}_+(\Sigma,G), \quad \text{and}\\
\cZ^{i+1}(\Sigma,G)&:=\cJ^{i+1}_+(\Sigma,G)/\cN^{i+1}(\Sigma,G).
\end{align*}  
Let us omit $(\Sigma,G)$ for brevity.  The following lemma summarizes results proven in \cite[\S 11]{Yu2001}.

\begin{lemma}
The quotient $\cW^{i+1}$ is a finite-dim\-en\-sion\-al $\resk$-vector space. Given $j,j'\in\cJ^{i+1}$, the commutator $[j,j']$ is contained in $G_{x,r_i}^i$, and $\phi^{i}\circ[-,-]$ defines a symplectic form on $\cW^{i+1}$.  Moreover, the quotient $\cH^{i+1}$ is a Heisenberg $p$-group with centre $\cZ^{i+1}$.
\end{lemma}

As elaborated carefully in \cite[3.25--26]{HakimMurnaghan2008}, we may canonically define the Heisenberg--Weil lift $\hat{\eta}$ of $\hat{\phi^i}$ to  the group $\mathrm{Sp}(\cW^{i+1}\ltimes\cH^{i+1}$.  Since $J^i=J^i(\Sigma,G)$ acts on $\cJ^{i+1}$ by conjugation, preserving the symplectic form, we have a map $J^{i}\ltimes \cJ^{i+1} \to \mathrm{Sp}(\cW^{i+1})\ltimes\cH^{i+1}$.  With respect to this map we set, for each $gj \in J^{i}\cJ^{i+1}$, 
$$
\kappa^i(gj) = \phi^i(g)\hat{\eta}(g,j).
$$
This is well-defined, and by \eqref{eq:product} gives an irreducible representation of $J^{i+1}$. %  for each $0\leq i <d$ one may associate to each character $\phi^{i}|_{J^i}$ and %$J^i \ltimes \cJ^{i+1}$ of $\hat{\phi^i}$ an irreducible representation $\kappa^{i}$ of $J^{i+1}(\Sigma,G)$. We will recall the necessary details where needed, in our proof of Proposition~\ref{prop:isotypic-restriction}.

Set $\kappa^d = \phi^d$.  Inflate each of the remaining representations $\kappa^i$ to a representation of $J(\Sigma,G):=J^d(\Sigma,G)$.  Using \eqref{eq:lift}, we may also inflate $\sigma$ to a representation of $J(\Sigma,G)$.
Set $\kappa_\Sigma=\bigotimes_{i=0}^d\kappa^i$ and define
$$
\lambda_\Sigma=\sigma\otimes\kappa_\Sigma.
$$

\begin{proposition}[{\cite[Theorem 7.5]{KimYu2017}}]\label{prop:kim-yu-types}
The representation $\lambda_\Sigma$ restricts irreducibly to $J(\Sigma,M)=J(\Sigma,G)\cap M$.  The pair $(J(\Sigma,M),\lambda_\Sigma)$ is an $\mathfrak{S}$-type for some finite set of inertial equivalence classes of $M$ supported on $M$.  Moreover, the pair $(J(\Sigma,G),\lambda_\Sigma)$ is a $G$-cover of $(J(\Sigma,M),\lambda_\Sigma)$.  
If $\KimYugroup=M^0_{x}$ then $\mathfrak{S}=\{[M,\pi_{\Sigma,M}]_M\}$ for some irreducible supercuspidal representation $\pi_{\Sigma,M}\hookrightarrow\cInd_{J(\Sigma,M)}^M\ \lambda_\Sigma$, and  $(J(\Sigma,G),\lambda_\Sigma)$ is an $[M,\pi_{\Sigma,M}]_G$-type.
\end{proposition}

In the special case that $\Sigma$ is a cuspidal datum and $J_G^0=G^0_x$ is maximal, the above construction gives types $(J,\lambda_\Sigma)$ which have the added property that if $(K,\tau)$ is an extension of $(J,\lambda_\Sigma)$ to a maximal compact-mod-centre subgroup of $G$, then $\cInd_K^G\tau$ is an irreducible supercuspidal representation whose inertial class lies is $[G,\pi]_G$ \cite{Yu2001}.

Upon restriction to the pro-$p$ radical $H_+(\Sigma,G)$ of $H(\Sigma,G)$, the representation $\lambda_\Sigma$ is $\theta_\Sigma$-isotypic,  where 
\begin{equation}\label{eq:def-theta}
\theta_\Sigma = \bigotimes_{i=0}^d\hat{\phi}^i,
\end{equation}
understanding this product relative to the factorization \eqref{eq:product2}. Following \cite{Stevens2005} we say that $(H_+(\Sigma,G),\theta_\Sigma)$ is the \emph{semisimple character} associated to $\Sigma$. In the case that  $\Sigma$ is cuspidal, we say that $\theta_\Sigma$ is a \emph{simple character}.  These characters encode the arithmetic information contained within the type and will be essential to our arguments in Section~\ref{sec:seville}.

We conclude this section with a result relating data that differ only in the choice of $(\KimYugroup,\sigma)$, for some $M_{x,0}\subseteq \KimYugroup \subseteq M_x$.

\begin{lemma}\label{lem:S-type-connected}
The restriction of $\kappa_\Sigma$ to $J_0(\Sigma,G)$ is irreducible, and the irreducible components of $\lambda_\Sigma|_{J_0(\Sigma,G)}$ are of the form $\sigma_0\otimes\kappa_\Sigma$, for some irreducible component $\sigma_0$ of $\sigma|_{M_{x,0}}$. Given some such irreducible component $\lambda_0=\sigma_0\otimes\kappa_\Sigma$, the pair $(J_0(\Sigma,G),\lambda_0)$ is an $\frak{S}$-type, for some finite set $\frak{S}$ of inertial equivalence classes, each of which is supported on $M$.
\end{lemma}

\begin{proof}
Given an irreducible component $\sigma_0$ of $\sigma|_{M_{x,0}}$, it follows immediately from the definition that
\[\Sigma_0=((\vec{\bG},\bM^0),(x,\{\iota\}),(M_{x,0},\sigma_0),\vec{r},\vec{\phi})
\]
is again a datum, and the corresponding type is precisely $(J_0(\Sigma,G),\lambda_0)$.  The uniqueness of the Heisenberg--Weil lift implies that $\kappa_\Sigma$ is an extension of $\kappa_{\Sigma_0}$, so restricts irreducibly; the second statement follows.
The remaining statements follow from 
Proposition~\ref{prop:kim-yu-types}.
\end{proof}

\section{Mackey theory}\label{sec:mackey}

Now, and for the remainder of the paper, we fix a cuspidal datum $\Sigma=(\vec{\bG},x,\sigma,\vec{r},\vec{\phi})$ of $G$ (with $J^0=G^0_x$) and freely use the additional terms defined in Definition~\ref{def:datum}.  Since $\Sigma$ is fixed, we omit it (and the pair $(\Sigma,G)$) from the notation.  We also fix a choice of irreducible supercuspidal representation $\pi$ of $G$ such that $(J,\lambda)$ is a $[G,\pi]_G$-type, where  $\lambda=\sigma\otimes\kappa$. 

Our goal is to investigate the following conjecture.

\begin{conjecture}[The unicity of types]\label{conj:unicity}
Suppose that $K$ is a maximal compact subgroup of $G$ and $(K,\tau)$ is a $[G,\pi]_G$-type. Then there exists a $g\in G$ such that $^gJ\subset K$ and $\Hom_K(\tau,\Ind_{^gJ}^K\ \lambda)\neq 0$.
\end{conjecture}

Fix a point $y\in\buil(G)$ such that $G_y$ is a maximal compact open subgroup of $G$; every maximal compact subgroup of $G$ arises in this way.  Our goal is to determine which, if any, of the irreducible components of $\pi|_{G_y}$ are $[G,\pi]_G$ types.  Since the restriction to $G_y$ of $\Pi$ is a sum of countably many copies of $\pi|_{G_y}$, 
we are free to work with $\Pi|_{G_y}$ instead. 

We begin with the Mackey decomposition \cite{Kutzko1977}
\[\Pi|_{G_y}=\bigoplus_{g:G_y\backslash G/J}\Ind_{^gJ\cap G_y}^{G_y}\ ^g\lambda.
\]
Instead of the component $\Ind_{^gJ\cap G_y}^{G_y}\ ^g\lambda$, we will prefer to consider its $g^{-1}$-conjugate
\[\tau(y,g):=\Ind_{J\cap G_{g^{-1}y}}^{G_{g^{-1}y}} \lambda.
\]
Note that, for a subrepresentation $\tau$ of $\tau(y,g)$, the pair $(G_{g^{-1}y},\tau)$ is a $[G,\pi]_G$-type %for some finite set $\frak{S}\subset\frak{B}(G)$ 
if and only if $(G_y,{}^g\tau)$ is.

In Section~\ref{sec:non-cusp}, we will work instead with slightly larger representations.

Fix an irreducible component $\sigma_0$ of $\sigma|_{G^0_{x,0}}$.  Then
$\Sigma_0 = (\vec{\bG},x,(G_{x,0}^0,\sigma_0),\vec{r},\vec{\phi})$ is again a cuspidal datum.  By Lemma~\ref{lem:S-type-connected} the corresponding pair $(J_0,\lambda_0)$ is an $\frak{S}$-type for $\frak{S}$ a finite set of inertial equivalence classes supported on $G$.
Since $\lambda$ occurs as a subrepresentation of $\Ind_{J_0}^{J} \lambda_0$, the representation $\Pi_0:=\cInd_{J_0}^G\ \lambda_0$ contains $\Pi$ as a direct summand.  In particular, any type occuring in $\Pi|_{G_y}$ also appears in $\Pi_0|_{G_y}$.

More precisely, for each $g\in G_y\backslash G / J$, $\tau(y,g)$ occurs as a subrepresentation of
\begin{equation}\label{eq:Mackey-0}
\bigoplus_{h \in (J\cap G_{g^{-1}y})\backslash J / J_0} 
\Ind_{J_0 \cap G_{g^{-1}y}}^{G_{g^{-1}y}} {}^h\lambda_0,
\end{equation}
where these latter summands can variously be thought of as twisted Mackey components ${}^h\tau_0(y,gh)$ of $\Pi_0|_{G_y}$, or of twisted Mackey components of ${}^h\Pi_0|_{G_y}$.

%
%Relating the Mackey decompositions of $\Pi_0|_{G_y}$ |and $\Pi|_{G_y}$ gives
%
% 
%%, and each component of the Mackey decomposition of $\Res_{G_y}\Pi_0$ is a sum of components of $\Res_{G_y}\Pi$.
%%\[\Pi_0|_{G_y}=\bigoplus_{g:G_y\backslash G/J_0}\Ind_{^gJ_0\cap G_y}^{G_y}\ ^g\lambda_0.\]
%%As above, we prefer to consider the 
%twisted Mackey components
%\[\tau_0(y,g):=\Ind_{J_0\cap G_{g^{-1}y}}^{G_{g^{-1}y}} \lambda_0,
%\]
%as $g$ runs over a set of representatives of $G_y\backslash G/J_0$.  
%It follows that for each such $g$, the set of irreducible components of $\tau_0(y,g)$ is precisely the same as that of
%$$
%\bigoplus_{h:(J\cap G_{g^{-1}y)}\backslash J/J_0}\ {}^h\tau(y,gh).
%$$

These decompositions into (twisted) Mackey components have a useful interpretation in terms of the building $\buil(G)$.  Namely, for each $g\in G$, $\tau(y,g)$ is a representation of the maximal compact open subgroup $G_{g^{-1}y}$; we think of it as attached to the point $g^{-1}y\in \buil(G)$ (or, when convenient, to the point $[g^{-1}y] \in \buil^{red}(G)$).  In fact, as $g$ ranges over a set of double coset representatives of $G_y\backslash G/J$, the representations are attached to distinct $J$-orbits of points in the $G$-orbit of $y$.  

Now the further $g^{-1}y$ is from the defining point $x$ of the datum, the smaller is the intersection $J\cap G_{g^{-1}y}$, and thus the greater the likelihood that $\tau(y,g)$ (which depends only on $\lambda|_{J\cap G_{g^{-1}y}}$) is not unique to a representation of $G$ determined by the type $(J,\lambda)$.

In more precise terms, let $[x,g^{-1}y]$ denote the geodesic in $\buil(G)$ from $x$ to $g^{-1}y$.  Then $G_x\cap G_{g^{-1}y} = G_{[x,g^{-1}y]} \subseteq G_z$ for any $z\in [x,g^{-1}y]$.  Since also $J\subseteq G_x$, we conclude that
$$
J\cap G_{g^{-1}y} \subseteq J\cap G_z
$$
for any $z\in [x,g^{-1}y]$.  In Sections~\ref{sec:non-cusp} and \ref{sec:seville}, we apply this reasoning to reduce the questions we are considering to local ones, depending only an a neighbourhood of $x$.  For example, we have the following immediate result.  

\begin{proposition} \label{prop:Mackeytypes}
Suppose that $g^{-1}y \in \buil(G)^{J}$., the set of fixed points in $\buil(G)$ under $J$.   Then for each irreducible component $\tau$ of ${}^g\tau(y,g)$, $(G_{y},\tau)$ is a $[G,\pi]_G$-type.
\end{proposition}

\begin{proof}
By definition, we have $g^{-1}y\in \buil(G)^{J}$ if and only if $J\subseteq G_{g^{-1}y}$.  In this case, $\tau(y,g)=\Ind_{J}^{G_{g^{-1}y}}\ \lambda$, which is a sum of $[G,\pi]_G$-types.
\end{proof}

Recall that as $J$ is a compact open subgroup of $G$, the image  $\buil(G)^{J}$ in $\buil^{\red}(G)$ is bounded and convex.  Since $J\subseteq G_x$, it follows that $x\in\buil(G)^{J}$. This fixed point set may meet the orbit of $[y]$ in $0$, $1$ or finitely many points; each of these points defines a different, non-conjugate type on $G_y$ occurring in $\pi$.  

In particular, we have demonstrated that the unicity of types, as stated in Conjecture \ref{conj:unicity}, is equivalent to the statement that the finitely many distinct types corresponding to points in $G\cdot y \cap \buil(G)^J$ are the only $[G,\pi]_G$-typical representations of $G_y$.

\begin{conjecture}[The unicity of types]\label{conj:unicity-2}
Suppose that $\tau(y,g)$ contains a $[G,\pi]_G$-type. Then $g^{-1}y\in\buil(G)^{J}$.
\end{conjecture}

%By Frobenius reciprocity, $\lambda|_{J\cap G_{g^{-1}y}}$ is a direct summand of $\Ind_{J_0\cap G_{g^{-1}y}}^{J\cap G_{g^{-1}y}}\ \lambda_0$, and so $\tau(y,g)$ is a direct summand of $\tau_0(y,g)$.

\section{Mackey components that intertwine with non-cuspidal representations}\label{sec:non-cusp}

We continue to work with the cuspidal datum $\Sigma= (\vec{\bG},x,\sigma,\vec{r},\vec{\phi})$ fixed in Section \ref{sec:mackey}, but focus on the $\mathfrak{S}$-type $(J_0,\lambda_0)=(J_0(\Sigma),\sigma_0\otimes \kappa_\Sigma)$.

Fix a Mackey component $\tau_0(y,g)$.  Write $\sH$ for the image of the map $J_0\cap G_{g^{-1}y}\rightarrow J_0/J_+=G_{x,0:0+}^0$, and write $\sG_x^0=G_{x,0:0+}^0$. In this section, we prove the following result.

\begin{theorem}\label{thm:non-cusp}
Suppose that there exists a proper parabolic subgroup $\sP$ of $\sG_x^0$ such that $\sH\subset\sP$. Then, for any irreducible component $\tau$ of $\tau_0(y,g)$, there exists a non-cuspidal irreducible representation $\pi'$ of $G$ such that $\Hom_{G_{g^{-1}y}}(\tau,\pi')\neq 0$. In particular, no irreducible component of $\tau_0(y,g)$ is a $\mathfrak{s}$-type for any $\frak{s}\in\mathfrak{S}$.
\end{theorem}

The proof of this result will occupy the remainder of this section, taking the form of a series of lemmata. Fix a proper parabolic subgroup $\sP$ of $\sG_x^0$ containing $\sH$, identified with a facet $\facet\subset \buil(G^0)$ adjacent to $x$.  Choose a point $z\in \facet$ for which $\sP$ identifies with $G_{z,0}^0/G_{x,0+}$.

Choose an apartment $\apart^0=\apart(\bG^0,\bS^0,F)\subset\buil(G^0)$ which contains both $x$ and $z$. Here, $\bS^0$ is a maximal split torus in $\bG^0$; let $\Phi^0=\Phi(\bG^0,\bS^0,F)$ be the corresponding rational root system. Then the point
\[z':=x+(x-z)\in\apart^0
\]
satisfies $\alpha(z'-x)=-\alpha(z-x)$ for all $\alpha\in\Phi^0$, which implies that $G_{z',0}^0/G_{x,0+}^0=\bar{\sP}$, the parabolic subgroup of $\sG_x^0$ opposite to $\sP$.

 By \cite[\S 6.3, Proposition 6.4]{MoyPrasad1996}, we may associate to $z$ an $F$-Levi subgroup $\bM^0$ of $\bG^0$; note that $\bM^0$ is also by this process associated to $z'$.  Moreover, the images in $\sG_x^0$ of $\bM^0(F)\cap G_{z,0}^0$ and $\bM^0(F)\cap G_{z',0}^0$ coincide with the Levi factor $\sM$ of $\sP$ given by $G_{z,0:0+}^0=G_{z',0:0+}^0$. Write $\sP=\sM\sN$ for the resulting Levi decomposition of $\sP$ in $\sG_x^0$. 
 
 Decompose $\sigma_0|_{\sM}$ into irreducible components as
\begin{equation}\label{eqn:decompose-sigma-on-M}
\sigma_0|_{\sM}=\bigoplus_{\xi\in\Xi}\xi
\end{equation}
where $\Xi$ is some index set.  For each $\xi\in\Xi$, fix a representative $(\sL_\xi,\zeta_\xi)$ of its cuspidal support; thus $\xi$ occurs as a summand of $\Ind_{\sL_\xi\sU_\xi}^\sM\ \zeta_\xi$ for any parabolic $\sQ_\xi=\sL_\xi\sU_\xi$ of $\sM$ with Levi factor $\sL_\xi$.

\begin{lemma}\label{lem:define-point-u-xi}
There exists a point $u_\xi\in\apart^0$ for which $G_{u_\xi,0:0+}^0=\sL_\xi$, giving rise to a twisted Levi sequence $\vec{\bL}_\xi\subset\vec{\bG}$, together with an $\vec{s}$-generic embedding $\iota$ of the corresponding buildings (relative to $u_\xi$) into $\buil(G)$, such that
\begin{equation}\label{eqn:J-xi-inside-J}
J_\xi:=G_{u_\xi,0}^0G_{u_\xi,s_0}^1\cdots G_{u_\xi,s_{d-1}}\subset J_0,
\end{equation}
where we have suppressed the notation $\iota$ on the left. 
\end{lemma}

%Now choose an apartment $\apart=\apart(\bG,\bS,F)$ of $\buil(G)$ containing $\apart^0$, and identify $z$ and $z'$ with their images there.

\begin{proof}
Since $\sL_\xi$ is a Levi subgroup of $\sG_x^0$ contained in $\sM$, there exists a facet $\facet_\xi\subset\apart^0$, whose closure contains both $x$ and $z'$, such that for any $u_\xi\in\facet_\xi$ we have $G_{u_\xi,0:0+}^0=\sL_\xi$. As above, associated to $u_\xi$ is a (proper) Levi subgroup $L_\xi^0=\bL_\xi^0(F)$ of $G^0$. It has the property that with respect to any embedding of $\buil(L_\xi^0)$ into $\buil(G^0)$, the intersection $L_\xi^0\cap G_{u_\xi,0}^0=(L_\xi^0)_{u_\xi,0}$ is a maximal parahoric subgroup of $L_\xi^0$ such that $(L_\xi^0)_{u_\xi,0:0+}=\sL_\xi$.

%Recall the seqence $\bG$ of twisted Levi subgroups of $\bG$. 
Now, using J.-L.~Kim and J.K.~Yu's construction as summarized in Remark~\ref{rem:gen-twisted-levi}, construct from $\bL_\xi^0$ the generalized twisted Levi sequence
$(\bL_\xi^0, \bL_\xi^1,\ldots, \bL_\xi^d=\bL_\xi)$.
%$\bL_\xi$ Following \cite[\S 2.4]{KimYu2017}, let $Z_s(\bL_\xi)^0$ %\bZ_\xi^0$  be the maximal $F$-split torus in the centre of $\bL_\xi^0$. For each $0\leq i\leq d$, let $\bL_\xi^i$ be the $\bG^i$-centralizer of $Z_s(\bL_\xi)^0$. Then, up to potential repetition if $\bL_\xi^i=\bL_\xi^{i+1}$ for some $i$, $\vec{\bL}_\xi$ is a twisted Levi sequence in $\bL_\xi:=\bL_\xi^d$ which is defined over $F$, and 
Note that $Z(\bL_\xi^0)/Z(\bL_\xi)$ is anisotropic over $F$ since  $Z(\bG^0)/Z(\bG)$ is.  Set $L_\xi^i=\bL_\xi^i(F)$.

We now choose a family of embeddings $\{\iota\}$ of the buildings $\buil(L_\xi^i)$ into $\buil(G^i)$ that is $\vec{s}$-generic and for which \eqref{eqn:J-xi-inside-J} holds as follows.  

Since $Z(\bL_\xi^0)/Z(\bL_\xi)$ is anisotropic, $\buil(L_\xi^0)$ embeds uniquely in $\buil(L_\xi^i)$ for each $i$; this defines points $u_\xi$ in each of these buildings.  An embedding $\iota$ of $\buil(L_\xi^i)$ into $\buil(G^i)$ is $(u_\xi,s_{i-1})$-generic if for each maximal split torus $\bS^i$ of $\bL_\xi^i$ such that $u_\xi\in\apart(\bL_\xi^i,\bS^i)$, then  for each root $\alpha\in\Phi(\bG^i,\bS^i,F)\backslash\Phi(\bL_\xi^i,\bS^i,F)$ we have that $U_{\alpha,\iota(u_\xi),s_{i-1}}^i=U_{\alpha,\iota(u_\xi),s_{i-1}+}^i$; in particular, genericity is the condition that $\iota(u_\xi)$ does not lie on any of the hyperplanes $H^i_{\psi,s_{i-1}}=\{y\in\apart(\bG^i,\bS^i,F) \mid \psi(y)=s_{i-1}\}$, for any affine root $\psi$ with gradient such an $\alpha$.

As $i$ ranges from $0$ to $d$, these hyperplanes $H^i_{\psi,s_{i-1}}$ subdivide each $\buil(G^i) \subset \buil(G)$, leaving a collection $O^i$ of open connected components.  We require a choice of embedding $\iota \colon \buil(L_\xi^0)\to\buil(G^0)$ so that for each $i$, $\iota(u_\xi)$ lies in some $O^i$ that contains $x$ in its closure. This is possible, via the argument of \cite[\S3.6]{KimYu2017}, since every  $\alpha\in\Phi(\bG^i,\bS^i,F)\backslash\Phi(\bL_\xi^i,\bS^i,F)$ is non-constant on $\facet_\xi \subset \buil(G^0)\subset \buil(G^i)$.

By construction, the geodesic $(x,\iota(u_\xi)]$ in $\buil(G)$ does not cross any hyperplane $H_{\psi,s_{i-1}}$ in $\buil(G^i)$ for any $i$, and so $U_{\alpha,\iota(u_\xi),s_{i-1}}^i\subset U_{\alpha,x,s_{i-1}}^i$ for all $\alpha\in\Phi(\bG^i,\bS^i,F)\backslash\Phi(\bL^i_\xi,\bS^i,F)$; moreover, we have $U_{\alpha,u_\xi,s_{i-1}}^i=U_{\alpha,x,s_{i-1}}^i$ for any $\alpha\in\Phi(\bL_\xi^i,\bS^i,F)$. This suffices to yield the inclusion $G_{\iota(u_\xi),s_{i-1}}^i\subset G_{x,s_{i-1}}^i$. %Since $u_\xi\in\facet_\xi$, which contains $x$ in its closure, we have $G_{u_\xi,0}\subset G_{x,0}$. 
Together, these inclusions imply \eqref{eqn:J-xi-inside-J}.
\end{proof}

%We fix $\iota$ once and for all, and omit it from our notation.
%With this, writing $\{\iota\}$ for the $\vec{s}$-generic diagram, one easily checks that we are able to associate a datum to each component $\xi\in\Xi$, namely:

With this choice of $\iota$ (which we fix once and for all, and omit from our notation), and letting $\zeta_\xi$ denote the inflation to $L_{\xi,0}^0$ of the cuspidal representation by the same name, it is easy to verify the following.

\begin{lemma}
The tuple
\[\Sigma_\xi:=((\vec{\bG},\bL_\xi^0),(u_\xi,\{\iota\}),(L^0_{\xi,0},\zeta_\xi),\vec{r},\vec{\phi})
\]
is a datum.
\end{lemma}

In particular, via $\Sigma_\xi$ we now associate to each irreducible component $\xi$ of $\sigma_0|_{\sM}$ a pair $(J_\xi,\lambda_\xi)$, where $J_\xi=J_0(\Sigma_\xi,G)$ as in \eqref{eqn:J-xi-inside-J} and $\lambda_\xi=\zeta_\xi\otimes\kappa_\xi$.  Note that $(J_\xi,\lambda_\xi)$ is an $\frak{S}$-type for some finite set $\frak{S}$ of inertial equivalence classes, all of which are supported on $L_\xi$; thus any irreducible subquotient of $\cInd_{J_\xi}^G\ \lambda_\xi$ is non-cuspidal.   Thus Theorem~\ref{thm:non-cusp} will follow if we show that, for any component $\tau$ of $\tau_0(y,g)$, there exists a $\xi\in\Xi$ such that $\tau$ is contained in an irreducible quotient of $\cInd_{J_\xi}^G \lambda_\xi$. 

A first step, which will be needed in the proof, is to understand the image of $J_\xi\cap G_{g^{-1}y}$ in $\sG_{x}^0$.  

\begin{lemma}\label{lem:image-parabolic-of-M}
For each $\xi\in\Xi$, the image $\sH_\xi$ of $J_\xi\cap G_{g^{-1}y}$ in $\sG_x^0$ is contained in a parabolic subgroup $\sQ_\xi$ of $\sM$ with Levi factor $\sL_\xi$.
\end{lemma}

\begin{proof}
Recall that $J_\xi\subset J_0$, and so $J_{+}\subset J_{\xi,+}$, which implies that $J_\xi/J_+\subset J_0/J_+=\sG_x^0$. Thus the image $\sH_\xi$ of $J_\xi\cap G_{g^{-1}y}$ is a subgroup of the image $\sH$ of $J_0\cap G_{g^{-1}y}$ in $\sG_x^0$, which by hypothesis lies in the proper parabolic subgroup $\sP$ of $\sG_x^0$. On the other hand, we chose $u_\xi$ so that
\[J_\xi/J_+=G_{u_\xi,0}^0/G_{x,0+}^0\subset G_{z',0}^0/G_{x,0}^0=\bar{\sP},
\]
the parabolic subgroup opposite to $\sP$. Therefore $\sH_\xi\subset\sM=\sP\cap\bar{\sP}$. In fact, since $G_{u_\xi,0}^0/G_{x,0+}^0$ is itself a parabolic subgroup $\sP'$ of $\sG_x^0$ with Levi factor $\sL_\xi$, we can set $\sQ_\xi=\sP'\cap\sP$, which is a parabolic subgroup of $\sM$ with Levi decomposition $\sQ_\xi=\sL_\xi\sN_\xi$, and deduce further that
\[\sH_\xi\subset\sL_\xi\sN_\xi=\sQ_\xi\subset\sM,
\]
as desired.
\end{proof}

The crucial technical step is for us to compare the Heisenberg--Weil representations $\kappa$ and $\kappa_\xi$.

\begin{proposition}\label{prop:isotypic-restriction}
The restriction of $\kappa$ to $J_\xi$ is $\kappa_\xi$-isotypic.
\end{proposition}

We will need the following basic lemma about Heisenberg representations.

\begin{lemma}\label{lem:stone-von-neumann}
Suppose $\cH_1\subset\cH_2$ are finite Heisenberg $p$-groups with common centre $\cZ$. Let $\phi$ be a non-trivial character of $\cZ$, and let $\eta_1,\eta_2$ be irreducible representations of $\cH_1,\cH_2$, respectively, both with central character $\phi$. Then $\eta_2|_{\cH_1}\simeq\eta_1^{\oplus[\cH_2:\cH_1]^{1/2}}$.
\end{lemma}

\begin{proof}
By the Stone--von Neuman theorem, the irreducible representations of a finite Heisenberg $p$-group $\cH_i$ with centre $\cZ$ are either characters or $[\cH_i:\cZ]^{1/2}$-dimensional, and the latter are uniquely determined by their non-trivial central characters.
\end{proof}

\begin{proof}[Proof of Proposition \ref{prop:isotypic-restriction}]
%The representation $\kappa$ of $J_0$ is constructed as a tensor product $\kappa_0\otimes\cdots\otimes\kappa_{d-1}$, where each $\kappa_i$ is the extension to $J_0$ of the Heisenberg--Weil lift arising by viewing $\phi^i$ as the central character of a certain Heisenberg group attached to $(G^i,G^{i+1})$.

%For the remainder of the proof, we fix some notation relative to a fixed $\xi\in\Xi$, so that we may write $u=u_\xi$. 

Recall that $\kappa$ is a tensor product $\kappa_0\otimes \kappa_1 \otimes \cdots \otimes \kappa_{d}$, where each $\kappa_i$ is obtained from the character $\phi^i$ via a Heisenberg-Weil lift (and $\kappa_d=\phi^d$). We wish to compare the constructions of the representations $\kappa_i$ and $\kappa_{\xi,i}$ for each $0\leq i \leq d-1$.  Fix an index $i$ throughout.

The essential difference between the groups $J$ and $J_\xi$ is the replacement of the point $x$ with $u=u_\xi$.  We therefore write $J_x=J$ and $J_u=J_\xi$ in order to keep track of and further emphasize this distinction. For each $w\in\{u,x\}$, we have the groups
\[
\cJ_w=(G^i,G^{i+1})_{w,(r_i,s_i)},\text{ and}\ 
\cJ_{w+}=(G^i,G^{i+1})_{w,(r_i,s_i+)}.
\]
Recall that we extend the character $\phi^i$ to a character $\hat{\phi}^i_w$ of $\cJ_{w+}$ and write $\cN_w=\ker\hat{\phi}^i_w$. We further define
\begin{align*}
\cH_w&=\cJ_w/\cN_w,\\
\cW_w&=\cJ_w/\cJ_{w+},\text{ and}\\
\cZ_w&=\cJ_{w+}/\cN_w.
\end{align*}
By \cite[\S 3.3]{HakimMurnaghan2008}, the character $\hat{\phi}_w^i$ of $\cJ_{w+}$ is characterized by the property that it coincides with $\phi^i$ on $G_{w,r_i}^i$, and is trivial on $(G^i,G^{i+1})_{w,(r_i+,s_i+)}$. Thus, in particular, $G_{w,r_i}^i\cN_w=\cJ_{w+}$, whence
\[\cJ_{w+}/\cN_w\simeq G_{w,r_i}^i/\ker\hat{\phi}^i.
\]
Fix a minimal Levi subgroup $C^i$ of $G^i$, arising as the centralizer of a maximal split torus $S^i=\bS^i(F)$ such that both $u$ and $x$ are contained in the apartment $\apart(\bG^i,\bS^i,F)$. Then,as noted previously, the Moy--Prasad filtration on $C^i$ is independent of the choice of $w\in\{u,x\}$. Since $\phi^i$ is a character of $G^i$, it is trivial on the derived subgroup of $G^i$, and hence on the root subgroups of $G^i$, which implies that
\[G_{w,r_i}^i/\ker\phi^i\simeq C_{r_i}^i/\ker\phi^i.
\]
This allows us to identify $\cZ_u$ with $\cZ_x$ and, moreover, allows us to observe that the restrictions of $\hat{\phi}_u^i$ and $\hat{\phi}_x^i$ coincide with that of $\phi^i$ under this identification. Thus for each $w\in \{u,x\}$, the character $\phi^i$ defines a symplectic structure on $\cW_w$, and the structure of a Heisenberg $p$-group on $\cH_w$, with centre $\cZ_w$. Our first step is to show that $\cW_u$ is a symplectic subspace of $\cW_x$.

Begin by considering our Levi subgroups $L^j:=L_\xi^j$ for $j\in\{i,i+1\}$. %we may also define the analogous objects for $L:=L_\xi$. 
Since $L_{x,r}=L_{u,r}$ for all $r\geq 0$ by construction of $L$, the following groups are in fact independent of the choice of $w\in\{u,x\}$:
\[
\cJ^L=(L^i,L^{i+1})_{w,(r_i,s_i)} \text{ and}\
\cJ_+^L=(L^i,L^{i+1})_{w,(r_i,s_i+)}.
\]
Set $\cW^L = \cJ^L/\cJ^L_+$.  The genericity of the embedding of $\buil(L^i)$ into $\buil(G^i)$ directly implies that $\cW_u\cong\cW^L$, as in the proof of \cite[Thm 7.5]{KimYu2017}. Because the root subgroups of $L^{i+1}$ are root subgroups of $G^{i+1}$, we deduce that $\cW^L$ is a subspace of the $\resk$-vector space $\cW_x$; the non-degeneracy of $\cW_u$ with respect to the symplectic form $\phi^i\circ[-,-]$ thus implies that $\cW_u$ is a symplectic subspace of $\cW_x$.

The inclusion of $\cJ^L$ into $\cJ_x$ thus induces an injective homomorphism $i:\cH_u\hookrightarrow\cH_x$, which restricts to an isomorphism of the centres of these Heisenberg $p$-groups. Let $(\eta_w,V_w)$ denote the Heisenberg representation of $\cH_w$ with central character $\phi^i$. By Lemma \ref{lem:stone-von-neumann}, $\eta_x$ becomes $\eta_u$-isotypic upon restriction to $\cH_u$.

Now consider the Weil lifts of each of these representations.  For $w\in\{u,x\}$, they are homomorphisms
\[\hat{\eta}_w:\mathrm{Sp}(\cW_w)\ltimes\cH_w\rightarrow\mathrm{Aut}(V_w).
\]
Note that $\hat{\eta}_w$ is characterized as the unique representation of $\mathrm{Sp}(\cW_w)\ltimes\cH_w$ extending $\eta_w$ (up to certain choices in small residual characteristic; see \cite[\S 2.3]{HakimMurnaghan2008}). The restriction of $\hat{\eta}_x$ to the subgroup $\mathrm{Sp}(\cW_u)\ltimes\cH_u$, being $\eta_u$-isotypic upon further restriction to $\cH_u$, must therefore be $\hat{\eta}_u$-isotypic.

Finally, the representation $\kappa_{w,i}$ of $J_w^{i+1}$ is obtained from $\phi^i$ and $\hat{\eta}_w$ by making the identification
\[J_w^{i+1}=G_{w,0}^0G_{w,s_0}^1\cdots G_{w,s_i}^{i+1}=J_w^i\cJ_w
\]
and then, for all $g\in J_w^i$ and all $j\in\cJ_w$, setting
\[\kappa_{w,i}(gj)=\phi^i(g)\hat{\eta}_w(g,j).
\]
Thus, upon restriction, the representation $\kappa_{x,i}$ is $\kappa_{u,i}$-isotypic, and the same is true of their respective inflations to $J_0=J^d_0(\Sigma,G)$ and $J_\xi=J^d_0(\Sigma_\xi,G)$.
\end{proof}

%Recall that $M^0=\sM^0(F)$, where $\sM^0$ is the $F$-Levi subgroup of $G^0$ associated to $z$, and that one has $M_{z,0:0+}^0=M_{x,0:0+}^0=\sM$.

With this, we are ready to complete the proof.

\begin{proof}[Proof of Theorem \ref{thm:non-cusp}]
Let $\tau$ be an irreducible component of $\tau_0(y,g)=\Ind_{J_0\cap G_{g^{-1}y}}^{G_{g^{-1}y}}\ \lambda_0$, where $\lambda_0=\sigma_0\otimes\kappa$. Frobenius reciprocity implies $\Hom_{J_0\cap G_{g^{-1}y}}(\tau,\sigma_0\otimes\kappa)\neq 0$.  We therefore have non-trivial intertwining between these representations on the smaller subgroup $J_0\cap G_{g^{-1}y}\cap M_{x,0}^0 G_{x,0+}$, whose image in $\sG_x^0$ lies in $\sM$.  Since $\sigma_0|_{\sM} = \oplus_{\xi\in\Xi}\xi$, we may choose $\xi$ for which
\[\Hom_{J_0\cap G_{g^{-1}y}\cap M_{x,0}G_{x,0+}}(\tau,\xi\otimes\kappa)\neq 0,
\]
where here we think of $\xi$ as a representation of  $M_{x,0}^0G_{x,0+}$ by inflation. 
By Lemmata \ref{lem:define-point-u-xi} and \ref{lem:image-parabolic-of-M}, we have
\[J_\xi\cap G_{g^{-1}y}\subset J_0\cap G_{g^{-1}y}\cap M_{x,0}G_{x,0+}.
\]
Moreover, by Proposition \ref{prop:isotypic-restriction}, the restriction to $J_\xi\cap G_{g^{-1}y}$ of $\kappa$ is $\kappa_\xi$-isotypic. Therefore we may further conclude that
\[\Hom_{J_\xi\cap G_{g^{-1}y}}(\tau,\xi\otimes\kappa_\xi)\neq 0.
\]
The cuspidal support of $\xi$ is $(\sL_\xi,\zeta_\xi)$; thus choosing the parabolic $\sQ_\xi$ of $\sM$ with Levi factor $\sL_\xi$ as in Lemma~\ref{lem:image-parabolic-of-M}, we have
\[0\neq\Hom_{\sM}(\xi,\Ind_{\sQ_\xi}^\sM\ \zeta_\xi\otimes\mathds{1})=\Hom_{\sQ_\xi}(\xi,\zeta_\xi\otimes\mathds{1}),
\]
with the latter identification following by Frobenius reciprocity.

Therefore it follows from Lemma \ref{lem:image-parabolic-of-M} that
\[\Hom_{J_\xi\cap G_{g^{-1}y}}(\tau,\zeta_\xi\otimes\kappa_\xi)\neq 0.
\]
%where we note that the inflation of the representation $\zeta_\xi\otimes \mathds{1}$ of $\sQ_\xi$ to $J_\xi\cap G_{g^{-1}y}$ being $\zeta_\xi$, it
Since $\tau$ is irreducible as a representation of $G_{g^{-1}y}$, applying Frobenius reciprocity reveals  that $\tau$ is a subrepresentation of
\[\pi_\xi:=\Ind_{J_\xi\cap G_{g^{-1}y}}^{G_{g^{-1}y}}\ \zeta_\xi\otimes\kappa_\xi.
\]
Note that by Mackey theory, we have 
\[\Res_{G_{g^{-1}y}}^G\cInd_{J_\xi}^G\ \zeta_\xi\otimes\kappa_\xi=\bigoplus_{h:G_{g^{-1}y}\backslash G/J_\xi}\Ind_{^hJ_\xi\cap G_{g^{-1}y}}^{G_{g^{-1}y}}\ ^h(\zeta_\xi\otimes\kappa_\xi).
\]
The summand for $h=1$ is exactly the representation $\pi_\xi$. It follows that $\tau$ is contained in $\cInd_{J_\xi}^G\ \zeta_\xi\otimes\kappa_\xi$, and hence, by Lemma~\ref{lem:induced-from-type-nsc} in some finitely generated subquotient, which admits an irreducible quotient $\pi'$ containing $\tau$. Since $(J_\xi,\zeta_\xi\otimes \kappa_\xi)$ is an $\mathfrak{S}$-type for a set of non-cuspidal inertial classes $\mathfrak{S}$ supported on $L_\xi$, this representation $\pi'$ is non-cuspidal.  Consequently, $(G_{g^{-1}y},\tau)$ is not a $[G,\pi]_G$-type.
\end{proof}

\section{Projection to the building of a twisted Levi subgroup}\label{sec:projection}

In Theorem \ref{thm:non-cusp}, we have a general procedure for showing that certain Mackey components $\tau(y,g)$ are atypical, based on the image $\sH$ of $G_{g^{-1}y}\cap J_0$ in $\sG^0_x$. In this section, we show that the hypotheses apply to many, but usually not all, Mackey components.

First, we establish a sufficient condition for a point $g^{-1}y\in\buil(G)$ to satisfy the hypotheses of Theorem \ref{thm:non-cusp} --- briefly, that the closest point on $\buil(G^0)$ to $g^{-1}y$ is not in $\tilde{[x]} :=\{[x]\}\times X_*(Z)\otimes_\mathbb{Z}\mathbb{R}$. This sufficient condition is satisfied by a large portion of the points $g^{-1}y\in\buil(G)$.

To state this precisely, note that the building $\buil(G)$ is a CAT(0) space, as is the image of $\buil(G^0)$ in $\buil(G)$. Thus any point $z\in\buil(G)$ has a unique closest point $z^0$ in $\buil(G^0)$; in this way, we may define a projection map $\proj_{\buil(G^0)}:\buil(G)\rightarrow\buil(G^0)$. We caution the reader that this projection map does not restrict to a projection map on each apartment. %can have some surprising properties; we will illustrate some of these throughout this section.

\begin{lemma}\label{lem:projection}
For any $z\in\buil(G)$, let $z^0= {\mathrm{proj}}_{\buil(G^0)}(z)$.  Then one has $G_z\cap G^0\subseteq G^0_{z^0}$.
\end{lemma}

\begin{proof}
%Set $z^0=\mathrm{proj}_{\buil(G^0)}(z)$, and l
Let $d$ denote the metric on $\buil(G)$.  Since each $g\in G^0\cap G_z$ fixes $z$ and acts by isometries on $\buil(G)$, we have $d(z,gz^0)=d(gz,gz^0)=d(z,z^0)$.  Since $g\in G^0$, $gz^0$ is again in $\buil(G^0)$, whence $gz^0=z^0$ by the uniqueness of $z^0$.
\end{proof}

The group $G_{g^{-1}y}\cap G_x$ fixes $x$ and $g^{-1}y$, hence the geodesic $[x,g^{-1}y]$ as well each element of $\tilde{[x]}$.  It may also fix more, such as if this geodesic passes through the interior of a chamber.  Set $\Gamma(x,g^{-1}y) = \buil(G)^{G_x\cap G_{g^{-1}y}}$ to be the full set of fixed points.  %By construction, $x\in \proj_{\buil(G^0)}\Gamma(x,g^{-1}y)$, as are all points in its equivalence class modulo $\buil^{\red}(G^0)$ (or $\buil^{\red}(G)$, since $Z^0/Z$ is anisotropic).

%Recall that $\sH$ is the image of $J_0\cap G_{g^{-1}y}$ in $\sG_x^0$.

\begin{proposition}\label{prop:non-cusp-hypothesis}
If  $\proj_{\buil(G^0)}\Gamma(x,g^{-1}y)\neq \tilde{[x]}$ %\times X_*(Z)\otimes_\Z\bbR$, %[x]\times X_*(Z^0)\otimes_\Z\bbR$ 
then $\sH$ is contained in a proper parabolic subgroup of $\sG_x^0$. %In particular, Theorem~\ref{thm:non-cusp} applies to $\tau_0(y,g)$.
\end{proposition}

\begin{proof}
Suppose $\proj_{\buil(G^0)}\Gamma(x,g^{-1}y)$ is strictly larger than $\tilde{[x]}$;  %\times X_*(Z)\otimes_\Z\bbR$;
then it meets a facet $\facet$ of $\buil(G^0)$ adjacent to $x$.  For any $z^0\in \facet$, $\sP = G^0_{z^0,0}/G^0_{x,0+}$ is a proper parabolic subgroup of $\sG^0_x$.

So choose $z^0\in \facet$ lying in the image of the projection of some point $z \in \Gamma(x,g^{-1}y)$.  Then $J\cap G_{g^{-1}y} \subset G_x\cap G_{g^{-1}y} = G_{\Gamma(x,g^{-1}y)}\subseteq G_{z}$.  Moreover, since the image of $J_0$ in $J_0/J_+$ is equal to the image of $G^0_{x,0}$, it follows that the image $\sH$ of $J_0\cap G_{g^{-1}y}$ in $J_0/J_+$ is contained in the image of $G^0_{x,0}\cap G_{z}$.  By Lemma~\ref{lem:projection}, $G^0\cap G_z\subseteq G^0_{z^0}$; intersecting further with the kernel of the Kottwitz homomorphism we deduce that $G^0_{x,0}\cap G_{z}\subseteq G^0_{z^0,0}$, as required.
\end{proof}

In particular, the hypotheses of this lemma are satisfied whenever $\proj_{\buil(G^0)}(g^{-1}y) \notin \tilde{[x]}$.  Applying Theorem \ref{thm:non-cusp} yields the following sufficient criterion.

\begin{corollary}
If $\proj_{\buil(G^0)}\Gamma(x,g^{-1}y) \neq \tilde{[x]}$, then $\tau(y,g)$ does not contain any $[G,\pi]_G$-typical representations.
\end{corollary}

%this tells us that suffices to show that  , as long as $\proj_{\buil(G^0)}\Gamma(x,g^{-1}y)$ is not contained in $[x]\times X_*(Z)\otimes_\Z\mathbb{R}$. That is, we are reduced to dealing only with those points $g^{-1}y$ which are contained in the pre-image under $\proj_{\buil(G^0)}$ of $[x]\otimes X_*(Z)\otimes_\Z\mathbb{R}$.

This hypothesis on $g^{-1}y$ is \emph{not} necessary in order for Theorem \ref{thm:non-cusp}
to hold, as illustrated in the example below.  Nevertheless, our next result gives a set of points in $\Delta=\proj_{\buil(G^0)}^{-1}\tilde{[x]}$ where Theorem \ref{thm:non-cusp} \emph{cannot} apply --- among them, the points on the buildings of suitably complementary twisted Levi subgroups.

Since $\Sigma$ is a cuspidal datum, we may assume that $\bG^0=C_\bG(\bZ^0)$, for $Z^0=\bZ^0(F)$ some minisotropic torus of $G$. Recall that we write $Z^0_{\mathrm{b}}$ for its maximal bounded subgroup---this is the group of $\frak{o}$-rational points of its lft-N\'{e}ron model. 

\begin{definition}\label{def:Gtilde}
Let $\tilde{Z}^0=\tilde{\bZ}^0(F)$ be a minisotropic torus contained in $G^0$ such that $\tilde{Z}_b^0\subset G_x^0$, and such that if $T=\langle Z^0,\tilde{Z}^0\rangle$, then $\mathsf{T}:=T_{0}/T_{0+}$ is a maximal minisotropic torus of $\sG_x^0$.  Then $\tilde{\bG}^0=C_G(\tilde{\bZ}^0)$ is a twisted Levi subgroup of $\bG$.  We call $\tilde{G}^0=\tilde{\bG}^0(F)$ a \emph{complementary twisted Levi subgroup} (of $G$ to $G^0$).
\end{definition}

As one example, one may choose for $\tilde{Z}^0$ an unramified torus such that $\tilde{Z}^0_{0:0+}$ is a maximal torus of $\sG^0_x$. % as Such a torus exists, by \cite{DeBacker}, though it may happen that $\tilde{Z}^0\supset \tilde{Z}^0$, which will not be interesting.  \textcolor{red}{Is this enough?  I need to think harder if I am convinced that $\buil(\tilde{G}^0)=[x]$, which would be OK.}
 Choose an embedding $\buil(\tilde{G}^0)\hookrightarrow\buil(G)$; it contains $x$.

\begin{proposition} \label{prop:twisted-levis}
If $g^{-1}y\in\buil(\tilde{G}^0)$, or more generally, if $g^{-1}y \in \buil(G)^{\tilde{Z}_0^0}$ then $\sH$ is  not contained in any proper parabolic subgroup of $\sG^0_x$.
\end{proposition}

\begin{proof}
Since $\buil(\tilde{G}^0)\subseteq \buil(G)^{\tilde{Z}_0^0}$, any 
such $g^{-1}y$ is fixed by $\tilde{Z}_0^0$.  By construction we also have $\tilde{Z}_0^0\subset G_{x,0}^0$. Thus $\tilde{Z}_{0}^0\subset G_{g^{-1}y}\cap G_{x,0}^0$ and hence $\tilde{Z}^0_{0:0+} \subset \sH$.  If $\mathsf{P}$ is a parabolic subgroup of $\sG_x^0$ %G_x^0/G_{x,0+}^0$ 
containing $\sH$, then it further contains $Z^0_{0:0+}$, hence the minisotropic maximal torus $\mathsf{T}$.  We conclude that $\mathsf{P}=\sG_x^0$.
  %the image of $J_0\cap G_{g^{-1}y}$. 
%Since both $\tilde{Z}^0_{0:0+}$ and $Z^0_{0:0+}$ are contained in $\mathsf{P}$, we conclude that $\mathsf{T}\subset\mathsf{P}$.  Since $\mathsf{T}$ is minisotropic, it follows that  
\end{proof}

By Proposition~\ref{prop:non-cusp-hypothesis}, for each such choice of $\tilde{Z}^0$, we have $\buil(\tilde{G}^0) \subset \Delta=\proj_{\buil(G^0)}^{-1}\tilde{[x]}$.  Moreover, for any point $g^{-1}y$ satisfying the hypothesis of Proposition~\ref{prop:twisted-levis}, Theorem \ref{thm:non-cusp} does not apply.  

We illustrate the regions of $\buil(G)$ introduced in Propositions~\ref{prop:non-cusp-hypothesis} and \ref{prop:twisted-levis} with an explicit example.

\begin{example}
Let $G=\mathbf{Sp}_4(F)$ and let $y$ be a vertex of $\buil(G)$.  Choose a twisted Levi sequence of $(G^0,G)$ of length two, where $G^0 \cong Z^0 \times \mathbf{SL}_2(F)$ with $Z^0 = \mathbf{Z}^0(F)$ an unramified $F$-anisotropic torus of rank $1$.  Suppose the vertex $x\in \buil(G^0)$ maps to a hyperspecial vertex of $\buil(G)$.
The rest of the datum is fully specified by a choice of character $\chi$ of $Z^0$ of depth $r$, and a cuspidal representation $\sigma$ of $\mathbf{SL}_2(\resk)$.  Let $\mathfrak{s}$ be the inertial class of the corresponding irreducible supercuspidal representation of $G$.  Note that here, $J = G^0_xG_{x,s_0}=J_0$ and $\buil(G)^{J}=\{x\}$.

We first consider the link of $x$ in $\buil(G)$, which is the closure $\Omega(x,1)$ of all chambers adjacent to $x$.  As all geodesics $[x,g^{-1}y]$ meet $\Omega(x,1)$, the goal is to classify the points in $\Omega(x,1)$ to which Theorem~\ref{thm:non-cusp} applies.

Let $\apart$ be an apartment containing $x$.  Set $\Omega_\apart = \Omega(x,1)\cap \apart$ and $\Omega_\apart^0=\Omega_\apart\cap \buil(G^0)$.  Then $\Omega_\apart^0$ contains $x$, as well as up to two chambers of $\buil(G^0)$.  We illustrate some possible configurations in Figure~\ref{figtocome}.

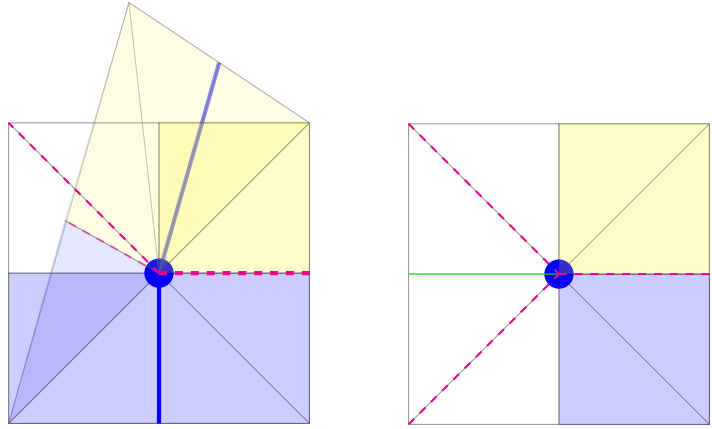
\begin{figure}[ht]
\begin{tikzpicture}[label/.style args={#1#2}{%
postaction={decorate,
decoration={markings,mark=at position #1 with \node #2;}}}]
% basic apartment
\draw[step = 2,gray,opacity = 0.8] (-2,-2) grid (2,2);
\draw[gray,opacity=0.8] (-2,-2) -- (2,2);
\draw[gray,opacity=0.8] (-2,2) -- (2,-2);
\draw[ultra thick,blue] (0,-2) -- (0,0);
\node[circle,fill=blue] at (0,0) {};
% piece of apartment jutting out:
\draw[gray,opacity=0.5] (-2,-2) -- (-0.4,3.6) -- (2,2);
\draw[gray,opacity=0.5] (-1.25,0.7) -- (0,0) -- (0.8,2.8);
\draw[gray,opacity=0.5] (0,0) -- (-0.4,3.6);
\draw[ultra thick,blue,opacity=0.5] (0,0) -- (0.8,2.8);

\draw[fill=yellow, opacity = 0.2] (0,0) -- (2,0) -- (2,2) -- (0,2) -- cycle;
\draw[fill=yellow, opacity = 0.1] (0,0) -- (-1.25,0.7) -- (-0.4,3.6) -- (2,2) -- cycle;

\draw[fill=blue, opacity = 0.2] (-2,0) -- (2,0) -- (2,-2) -- (-2,-2) -- cycle;
\draw[fill=blue, opacity = 0.1] (-2,-2) -- (-1.25,0.7) -- (0,0) -- cycle;

% other twisted Levis
\draw[ultra thick,magenta,dashed] (0,0) -- (2,0);
\draw[thick,magenta,dashed, opacity=0.5] (0,0) -- (-1.25,0.7);
\draw[thick,magenta,dashed] (0,0) -- (-2,2);
\end{tikzpicture}
\hspace{1cm}
\begin{tikzpicture}[label/.style args={#1#2}{%
postaction={decorate,
decoration={markings,mark=at position #1 with \node #2;}}}]
% basic apartment
\draw[step = 2,gray,opacity = 0.8] (-2,-2) grid (2,2);
\draw[gray,opacity=0.8] (-2,-2) -- (2,2);
\draw[gray,opacity=0.8] (-2,2) -- (2,-2);
\node[circle,fill=blue] at (0,0) {};
% projection to $x$
\draw[green] (-2,0) -- (0,0) -- cycle;
\draw[fill=yellow, opacity = 0.2] (0,0) -- (2,0) -- (2,2) -- (0,2) -- cycle;
\draw[fill=blue, opacity = 0.2] (0,0) -- (2,0) -- (2,-2) -- (0,-2) -- cycle;
% buildings of twisted levis
\draw[magenta,thick,dashed] (-2,2) -- (0,0);
\draw[magenta,thick,dashed] (2,0) -- (0,0);
\draw[magenta,thick,dashed] (-2,-2) -- (0,0);
\end{tikzpicture}

\caption{\label{figtocome} Neighbourhoods of $x$ in three apartments.  The dark blue line represents their intersection with $\buil(G^0)$.  Chambers of $\Delta$ are white and for all complementary twisted Levis $\tilde{G}^0$, the intersection of these apartments with $\buil(\tilde{G}^0)=\buil(G)^{\tilde{Z}^0_0}$  are indicated by magenta dashed lines.    The points whose projection onto $\buil(G^0)$ lie in some chamber of $\buil(G^0)$ are indicated by colours: blue, yellow, and green, corresponding to distinct chambers of $\buil(G^0)$ (not all shown).}
\end{figure}

In the simplest case, $\Omega_\apart^0$ contains two chambers of $\buil(G^0)$, and then the restriction to $\Omega_\apart$ of $\proj_{\buil(G^0)}$
coincides with the orthogonal projection map. Thus Proposition~\ref{prop:non-cusp-hypothesis} applies to all points save those on the orthgonal complement of $\Omega_\apart^0$ through $x$, which is $\Delta \cap \apart$.  This set coincides with $\Omega_\apart \cap \buil(\tilde{G}^0)$, where $\tilde{G}^0=\mathbf{SL}_2(F)\times Z^0$ is a complementary twisted Levi.  

Suppose $[x,g^{-1}y]$ meets a vertex $z$ of $\Delta$ adjacent to $x$; then $G_{g^{-1}y}\cap J \subseteq G_z \cap J \subseteq Z^0_{0+}\times \mathbf{SL}_2(\R)$.  Thus by multiplying $\chi$ by a depth-zero character of $Z^0$ we obtain the datum of a new supercuspidal representation containing ${}^g\tau(y,g)$, whence this latter cannot contain an $\mathfrak{s}$-type.  We conclude that if $g^{-1}y \in \apart$ then $\tau(y,g)$ contains an $\mathfrak{s}$-type if and only if $g^{-1}y=x$.

When $\Omega_\apart^0$ contains only one chamber $\facet$ of $\buil(G^0)$ the geometry becomes more complex.  Let $H$ denote the part of $\Omega_\apart$ which is the union of lines orthogonal to $\facet$; since these are geodesics, we conclude they represent the projection onto $\buil(G^0)$.  However, if the nearest point in $\Omega_\apart^0$ to some $z$
is $x$, then it need not follow that $x=\proj_{\buil(G^0)}(z)$ --- it may be that there is a closer point in $\buil(G^0)$ to $z$ which lies in a different apartment.  

We illustrate this phenomenon in the first figure of  Figure~\ref{figtocome}, where the two overlapping apartments depicted differ by conjugation by an element of the appropriate root subgroup.  One sees that $\Delta$ contains chambers of $\buil(G)$,  in addition to rays corresponding to buildings of twisted Levi subgroups as in Proposition~\ref{prop:twisted-levis}.

Finally, when $\Omega_\apart^{0}=\{x\}$, no aspect of the map $\proj_{\buil(G^0)}$ can be inferred from $\apart$.  As illustrated in the second figure of Figure~\ref{figtocome}, some chambers of $\Omega_\apart$ will lie in $\Delta$, but not all.  

In fact, in both of the examples illustrated in Figure~\ref{figtocome}, it can be shown that all chambers of $\Omega_\apart$ satisfy the hypotheses of Theorem~\ref{thm:non-cusp}, with those $z$ lying in chambers of $\Delta$  yielding particularly small subgroups $\sH$.  The only points excluded from the application of Theorem~\ref{thm:non-cusp} are those of $\buil(\tilde{G}^0)$, for various complementary twisted Levi subgroups $\tilde{G}^0$.
\end{example} 

The subject of the next section is to provide a tool (Theorem~\ref{thm:seville}) which may be applied to the Mackey components corresponding to points in $\buil(\tilde{G}^0)$, or more generally in $\Delta$.  For example, when $r>2$, then it can be shown that Theorem~\ref{thm:seville} applies to  all $z \in (\Omega(x,2)\setminus \Omega(x,1) )\cap \apart \cap \buil(\tilde{G}^0)$, where $\Omega(x,2)$ is the link of $\Omega(x,1)$, and $\apart$ is one of the apartments depicted in the example above.  In such cases we can conclude that for all $g$ for which $z\in [x,g^{-1}y]$,  $\tau(y,g)$ cannot contain an $\mathfrak{s}$-type.

\section{Points fixed by subgroups of $J$ and intertwining simple characters}\label{sec:seville}

\newcommand{\spt}{u}

We continue to work with the cuspidal datum $\Sigma= (\vec{\bG},x,\sigma,\vec{r},\vec{\phi})$ fixed in Section \ref{sec:mackey}. In particular, recall that the subgroup $H_+ = G^0_{x,0+}G^1_{x,s_0+}\cdots G^d_{x,s_{d-1}+}$ of $J=G^0_xG^1_{x,s_0}\cdots G^d_{x,s_{d-1}}$ carries the simple character $\theta_\Sigma$ obtained from $\vec{\phi^i}$ by appropriate inflation and restriction, as in \eqref{eq:def-theta}.

In this section, we establish another condition on $g^{-1}y$ which suffices to show that the Mackey component $\tau(y,g)$ does not contain any $[G,\pi]_G$-typical representations; this method is particularly targeted at points lying in the subset $\Delta$ constructed in the previous section.

\begin{lemma}\label{lem:cJ-normalizes-Ht}
The subgroup $H_+$ of $J$ is normal; hence for each $t>0$ and $1\leq i \leq d$, $J^i$ normalizes $H_t = H_+\cap G_{x,t}$.
\end{lemma} 

\begin{proof}
Since $J^i =J\cap G^i \subseteq G^k_{x}$ for all $k\geq i$ and $G^i_{x,t}\lhd G^i_x$ for all $t, t+ \geq 0$, to show the first statement we have only to consider $g \in G^i_{x,s_{i-1}}$ and $h\in G^k_{x,s_{k-1}+}$ for some pair $k< i$. 
Recall that a property of the Moy-Prasad filtration is that $[G^k_{x,r},G^k_{x,s}]\subseteq G^k_{x,r+s}$ for all $r,s\in \mathbb{R}\cup \mathbb{R}+$, $r,s\geq 0$; see for example \cite[\S2]{HakimMurnaghan2008}.
% and let $j$ be the largest index such that $s_{j-1}< t$.  We show that each factor of $J=G^0_xG^1_{x,s_0}\cdots G^d_{x,s_{d-1}}$ normalizes $H_t$. Let $G^k_{x,r}$ denote one of the factors of the product $H_t = G^j_{x,t}G^{j+1}_{x,s_j+}\cdots G_{x,s_{d-1}+}$ (that is, allowing $r\in \R\cup \R+$).    
%We show that each factor of $J=G^0_xG^1_{x,s_0}\cdots G^d_{x,s_{d-1}}$ normalizes $H_+$. Let $G^k_{x,r+}$ denote one of the factors of the product $H_+ = G^0_{x,0+}G^{1}_{x,s_0+}\cdots G_{x,s_{d-1}+}$. 
%\footnote{\textcolor{red}{insert $s_{-1}=0$ convention earlier?}}
Therefore since $G^k \subseteq G^i$ and $s_{k-1}+>0$,  $[h,g]\subseteq G^{i}_{x,s_{i-1}+}$, whence $g^{-1}hg \in hG^{i}_{x,s_{i-1}+} \subseteq H_+$.   
 Since $J^i\subseteq J \subseteq G_x$ and for any $t>0$ and $G_{x,t}$ is normal in $G_x$, the lemma follows.
\end{proof}

Recall that $Z^i$ denotes the center of $G^i$.  For each $t>0$ let $i$ be the largest index such that $s_{i-1}<t$, so that $H_t = G^i_{x,t}G^{i+1}_{x,s_i+}\cdots G_{x,s_{d-1}+}$.
Let
\[\Theta_t=\buil(G)^{H_{t+}}\backslash\buil(G)^{Z_t^i}
\]
be the set of points of $\buil(G)$ fixed by $H_{t+}$ but not by the subgroup $Z^i_t$ of $H_t$.  This is empty if $t>s_{d-1}$, since then $i=d$ and $Z_t$ fixes $\buil(G)$ pointwise.  Let $$\Theta=\bigcup_{0<t\leq s_{d-1}}\Theta_t.$$

\begin{theorem}\label{thm:seville}
Suppose that $g\in G$ is such that $\Theta\cap[x,g^{-1}y]\neq\emptyset$. Then no irreducible subrepresentation of the Mackey component $\tau(y,g)$ may be a $[G,\pi]_G$-type.
\end{theorem}

We offer a geometric interpretation of this hypothesis in Section~\ref{sec:unicity}.

\begin{proof}
Given $g\in G$ such that $\Theta\cap[x,g^{-1}y]\neq\emptyset$, there exists  $0<t\leq s_{d-1}$ such that we may choose %such that $\Theta_t\cap[x,g^{-1}y]\neq\emptyset$, and choose 
$\spt \in\Theta_t\cap[x,g^{-1}y]$.  Let $0\leq i < d$ be maximal with respect to the property that $s_{i-1}<t$.  

Since $\spt\in[x,g^{-1}y]$, we have $J\cap G_{g^{-1}y}\subset J\cap G_\spt$ and $H_t\cap G_{g^{-1}y}\subset H_t\cap G_\spt$. Since $\spt\in\buil(G)^{H_{t+}}\backslash\buil(G)^{Z_t^i}$, we have $H_{t+}\subset G_\spt$ but $Z_t^iH_{t+}\not\subset G_\spt$. Therefore we may choose an element $\bz\in Z_t^i$ such that $\bz H_{t+}\not\subset G_\spt$. Since $t\in(s_{i-1},s_i]$, the quotient $H_t/H_{t+}$ is isomorphic to the abelian group $G_{x,t:t+}^i$.

Let $\mu$ be a character of $H_t/H_{t+}$ which is trivial on $(H_t\cap G_\spt)H_{t+}$ but non-trivial on $\bz H_{t+}$. Inflate $\mu$ to a character of $H_t$. Write $\theta_t=\theta_\Sigma|_{H_t}$, and set $\theta'=\theta_t\mu$. Then $\theta_t$ and $\theta'$ are characters of $H_t$ with the following properties:
\begin{enumerate}[(i)]
\item $\theta_t=\theta'$ on $H_t\cap G_\spt\supset H_t\cap G_{g^{-1}y}$;
\item $\theta_t=\theta'$ on $H_{t+}$; and
\item $\theta_t(\bz)\neq\theta'(\bz)$.
\end{enumerate}
Now suppose that $\tau$ is an irreducible subrepresentation of the Mackey component $\tau(y,g)$, so that by Frobenius reciprocity we have 
\[0\neq\Hom_{G_{g^{-1}y}}(\tau,\Ind_{J\cap G_{g^{-1}y}}^{G_{g^{-1}y}}\lambda)=\Hom_{G_{g^{-1}y}\cap J}(\tau,\lambda).
\]
It follows that $\tau$ and $\lambda$ intertwine on the smaller subgroup $G_{g^{-1}y}\cap H_t$, where $\lambda|_{H_t}$ is $\theta_t$-isotypic. Applying Frobenius reciprocity again, we have
\begin{equation}\label{E:xitheta}
0\neq\Hom_{G_{g^{-1}y}\cap H_t}(\tau,\theta_t)=\Hom_{H_t}(\Ind_{G_{g^{-1}y}\cap H_t}^{H_t} \tau,\theta_t).
\end{equation}
By (i), we can replace $\theta_t$ by $\theta'$ on $G_{g^{-1}y}\cap H_t$, so that \eqref{E:xitheta} is equivalent to the statement
\begin{equation}\label{E:xithetap}
0\neq\Hom_{G_{g^{-1}y}\cap H_t}(\tau,\theta')=\Hom_{H_t}(\Ind_{G_{g^{-1}y}\cap H_t}^{H_t} \tau,\theta').
\end{equation}
Now, and for the remainder of the proof, suppose to the contrary that $(G_{g^{-1}y},\tau)$ is a $[G,\pi]_G$-type.  Then by \cite[Proposition 5.2]{BushnellKutzko1998}, any irreducible subquotient of $\pi_\tau:=\cInd_{G_{g^{-1}y}}^G \tau$ is isomorphic to the twist of $\pi$ by some unramified character of $G$.

We recognize the representation $\Ind_{G_{g^{-1}y}\cap H_t}^{H_t} \tau$ appearing in \eqref{E:xitheta} and \eqref{E:xithetap} as a Mackey component of $\Res_{H_t}^G \pi_\tau$, which is by the preceding a direct sum of copies of $\Res_{H_t}^G \pi$. It therefore follows that the characters $\theta_t$ and $\theta'$ both occur in $\Res_{H_t}^G \pi$. We will show that this cannot be the case.

Given that %Choose a Mackey component of
\[\Res_{H_t}^G\ \pi=\bigoplus_{\cosrepa:H_t\backslash G/J}\Ind_{H_t\cap{}^\cosrepa J}^{H_t}\ ^\cosrepa\lambda,
\]
we may choose $\cosrepa\in G$ such that the corresponding component contains $\theta'$.  We then have
\[0\neq\Hom_{H_t}(\theta',\Ind_{H_t\cap{}^\cosrepa J}^{H_t}\ ^\cosrepa\lambda)=\Hom_{H_t\cap{}^\cosrepa J}(\theta',{}^\cosrepa\lambda).
\]
Therefore, on the smaller subgroup $H_t\cap{}^\cosrepa H_t$, where $^\cosrepa\lambda$ is $^\cosrepa\theta_t$-isotypic, we have 
\begin{equation}\label{E:thpth}
\Hom_{H_t\cap{}^\cosrepa H_t}(\theta',{}^\cosrepa\theta_t)\neq 0.
\end{equation}
In other words, $\cosrepa\in G$ intertwines the two characters $\theta_t$ and $\theta'$ on $H_t$.

To complete the proof, we need to show that we may choose $\cosrepa \in G^i$.  To do so, we carry out a variant of an inductive argument appearing in \cite{Hakim2018} and \cite{Yu2001}, which is central to the proof that $(J,\lambda)$ is a $[G,\pi]_G$-type.

Suppose that $i\leq j<d$ and $\cosrepa\in G^{j+1}$ intertwines $\theta_t$ and $\theta'$ on $H_t$. That is, $\cosrepa$ satisfies
\[\Hom_{H_t\cap{}^\cosrepa H_t}(\theta',{}^\cosrepa\theta_t)\neq 0.
\]
Since $\theta_t=\theta'$ on $H_{t+}$, when we restrict to $H_{t+}\cap {}^\cosrepa H_{t+}$ we can replace $\theta'$ by $\theta_t$ to conclude that
\[\Hom_{H_{t+}\cap{}^\cosrepa H_{t+}}(\theta_t,{}^\cosrepa\theta_t)\neq 0.
\]
This intertwining implies that
\[[\cosrepa^{-1},H_{t+}]\cap H_{t+}\subset\ker\theta_t\subset\ker\theta.
\]
Recall that $H_+$ contains the subgroups $\cJ^{\ell}$ of \eqref{eq:def-script-J}.  Since $t\leq s_i\leq s_j$, we have that $\cJ_+^{j+1}\subset H_{s_j+}\subset H_{t+}$, whence
\[[a^{-1},\cJ_+^{j+1}]\cap\cJ_+^{j+1}\subset\ker\theta.
\]
Since $\cosrepa\in G^{j+1}$ and $\cJ_+^{j+1}\subset G^{j+1}$, this is a subgroup of $G^{j+1}_{\der}\cap\cJ_+^{j+1}$. By \cite[Lemma 3.9.1(5)]{Hakim2018}, the character $\theta$ coincides with the character $\hat{\phi}_j$ on this intersection.  Consequently $\hat{\phi}_j$ is also trivial on $[a^{-1},\cJ_+^{j+1}]\cap\cJ_+^{j+1}$, which implies $\cosrepa$ intertwines $\hat{\phi}_j$, that is,
\[\Hom_{\cJ_+^{j+1}\cap{}^\cosrepa\cJ_+^{j+1}}(\hat{\phi}_j,{}^\cosrepa\hat{\phi}_j)\neq 0.
\]
Applying \cite[Thm 9.4]{Yu2001}, we infer that $\cosrepa\in\cJ^{j+1}G^j\cJ^{j+1}$. Thus there exist $j_1,j_2\in\cJ^{j+1}$ and $\cosrepb\in G^j$ such that $\cosrepa=j_1\cosrepb j_2$. We wish to show that $\cosrepb$ intertwines $\theta'$ and $\theta_t$ on $H_t$, which is to say that $\theta'$ and $^\cosrepb\theta_t$ agree on $H_t\cap{}^\cosrepb H_t$.

Since in general $\cJ^{j+1}\not\subset H_t$, the representation $\lambda$ of $J$ is not $\theta_t$-isotypic on this subgroup.  However, in what follows, for each $h\in H_t$, we may make the identification $\lambda(h)=\theta_t(h)$, with $\theta_t(h)$ being viewed as a scalar action on the space of $\lambda$. %but $\cJ^{j+1}\subset J$, and the representation $\lambda$ of $J$ restricts to a $\theta_t$-isotypic representation on $H_t$. 

So let $h\in H_t\cap{}^\cosrepb H_t$. Noting that $\cosrepb^{-1}h\cosrepb\in H_t$ yields
\begin{align*}
^\cosrepb\theta_t(h)&=\theta_t(\cosrepb^{-1}h\cosrepb)=\theta_t(j_2\cosrepa^{-1}j_1hj_1^{-1}\cosrepa j_2^{-1})\\
&=\lambda(j_2\cosrepa^{-1}j_1hj_1^{-1}\cosrepa j_2^{-1})\\
&=\lambda(j_2)\ {}^\cosrepa\lambda(j_1hj_1^{-1})\lambda(j_2^{-1}).
\end{align*}
Since $j_1\in \cJ^{j+1}\subset J^{j+1}$ normalizes $H_t$ by Lemma \ref{lem:cJ-normalizes-Ht}, we have
\[j_1hj_1^{-1}\in{}^{j_1}(H_t\cap{}^\cosrepb H_t)=H_t\cap{}^{j_1\cosrepb}H_t=H_t\cap{}^\cosrepa H_t,
\]
whence $^\cosrepa\lambda(j_1hj_1^{-1})$ acts by the scalar $^\cosrepa\theta_t(j_1hj_1^{-1})$, which commutes with $\lambda(j_2)$. Since $\theta'$ and $^\cosrepa\theta_t$ agree on $H_t\cap{}^\cosrepa H_t$, we conclude that
\[^\cosrepb\theta_t(h)={}^\cosrepa\theta_t(j_1hj_1^{-1})=\theta'(j_1hj_1^{-1}).
\]
Now recall that $\mu=\theta_t^{-1}\theta'$ is a character of $H_t$ which is trivial on $H_{t+}$. When $t \neq s_i$, we have $J_t/J_{t+}=H_t/H_{t+}$, and so $\mu$ inflates to a character of $J_t$. If $t=s_i$, then $H_{s_i}/H_{s_i+}=G_{x,s_i}^iG_{x,s_i+}^{i+1}/G_{x,s_i+}^{i+1}$, which is a subgroup of the abelian group $J_{s_i}/J_{s_i+}=G_{x,s_i:s_i+}^{i+1}$. Thus we can extend $\mu$ to a character of $J_{s_i}$ which is trivial on $J_{s_i+}$.

Consider now the representation of $J_t$ given by $\lambda'=\mu\otimes\Res_{J_t}^J \lambda$. Upon restriction to $H_t$, this representation is $\mu\theta_t=\theta'$-isotypic. Since $t\leq s_i\leq s_j$, we have that $\cJ^{j+1}\subset J_t$, whence
\[\theta'(j_1hj_1^{-1})=\lambda'(j_1)\theta'(h)\lambda'(j_1^{-1})=\theta'(h).
\]
Therefore we have found an element $\cosrepb\in G^j$ for which $\theta'$ an $^\cosrepb\theta_t$ agree on $H_t\cap{}^\cosrepb H_t$.

Hence by induction, we may assume that we are given an $\cosrepa\in G^i$ such that $\theta'$ and $^\cosrepa\theta_t$ agree on $H_t\cap{}^\cosrepa H_t$.   Since $Z^i$ is central in $G^i$, we therefore have that $Z_t^i \subseteq H_t\cap{}^\cosrepa H_t$ and that the characters $\theta'$ and $\theta_t$ agree on $Z_t^i$, contradicting (iii) above.

Therefore no $\cosrepa\in G$ can intertwine $\theta_t$ and $\theta'$, which implies in turn that $\tau$ cannot be a $[G,\pi]_G$-type.
\end{proof}

\begin{remark}
A natural question is to ask to what extent the above proof by contradiction may be made constructive. That is: can one produce an irreducible representation $\pi'$ of $G$, not isomorphic to an unramified twist of $\pi$, such that $\pi'|_{G_{g^{-1}y}}$ contains $\tau(y,g)$? In general, the character $\mu$ arising during the proof need not extend to a character of $G^i$. However, should $\mu$ extend to $G^i$---for example, when $\mu$ is trivial on $[G^i, G^i]_{x,t}$%as happens, for example, when $[G^i,G^i]\cap J\cap G_{g^{-1}y}\subset J\cap G_{g^{-1}y}$
---then we may realize $\pi'$ as a supercuspidal representation obtained by replacing the character $\phi^i$ in the datum $\Sigma$ with $\mu\phi^i$.
\end{remark}

\section{Implications towards the unicity of types}\label{sec:unicity}

Recall that, in Conjecture \ref{conj:unicity-2}, we reformulated the unicity of types as the assertion that, if $\tau$ is an irreducible $[G,\pi]_G$-typical representation of a maximal compact subgroup $G_y$, then we must have $\tau$ is a subrepresentation of $^{g}\tau(y,g)$ for some $g\in G$ such that $g^{-1}y\in\buil(G)^J$.

In this section, we describe how our results reduce this conjecture to two explicit questions; indeed, the generality of our results strongly suggests that the unicity of types is inextricably linked to these two problems, each of which appears to be difficult in its own right.

In Section \ref{sec:projection}, we considered the set $\Delta$ consisting of points $z\in\buil(G)$ such that $\proj_{\buil(G^0)}(z)\in \tilde{[x]}$; on all points $g^{-1}y$ in the complement of $\Delta$ in $\buil(G)$,  Theorem~\ref{thm:non-cusp} ensures that $\tau(y,g)$ contains no $[G,\pi]_G$-types. 
%We also identified some twisted Levi subgroups $\tilde{G}^0$ of $\buil(G)$ such that $\buil(\tilde{G}^0)\subseteq\Delta$.

Our first expectation is that, under a mild hypothesis on $s_0$, it is possible to apply Theorem \ref{thm:seville} to the points $g^{-1}y$ in $\Delta$ outside of a compact neighbourhood of $[x]$ in $\buil^{\red}(G)$ (which may be strictly larger than $\buil^{\red}(G)^J$).

The reasoning is as follows.  Following the notation of Theorem~\ref{thm:seville}, note that for each $i$, $Z_0^i$ fixes $\buil(G^i)$ as a subset of $\buil(G)$; thus the filtration subgroup $Z^i_t$ fixes a $G^i$-invariant convex neighbourhood of $\buil(G^i)$. Since $Z^i_t\subset G^i_t$, we have
$
H_{t+}\subseteq Z^i_tH_{t+}.
$
Thus 
$$
\buil(G)^{H_{t+}}\supseteq \buil(G)^{Z^i_tH_{t+}}
$$ 
and their images in the reduced building are bounded convex neighbourhoods of $[x]$ since the groups  are compact open and contained in $G_x$.  Note that $\Theta_t$ is exactly their set-theoretic difference.  The set $\Theta_t$ can be empty (for example, whenever $Z^i_t = Z^i_{t+}\subset H_{t+}$) and it never meets $\buil(G^i)$.  One expects that as $t$ varies over an interval of length $1$, the regions $\Theta_t$ cover the radial directions from each of the buildings $\buil(G^i)$, in the sense of meeting the lines $[x,g^{-1}y]$ for $g^{-1}y$ sufficiently far from some $\buil(G^i)$.

In particular, by this reasoning, when $i=0$ and $s_0>1$ we expect $\Theta^0 := \cup_{0<t<s_0}\Theta_t$ to define an annulus about $x$ with axis $\buil(G^0)$, that is, a set with compact image in $\buil^{\red}(G)$ which meets every ray in $\Delta$ emanating from $x$.  In consequence, for every $g$ such that $g^{-1}y \in \Delta$ is not on the interior of this annulus, we would have that $\tau(y,g)$ cannot contain a $[G,\pi]_G$-type.

That said, in general this statement hinges on understanding the growth of the fixed-point sets $\buil(G)^{Z_t^0}$ as  $t>0$ increases.  When $Z^0=G^0$ is a maximal torus, we had reduced the question to the relatively simpler case of $t=0$ in  \cite{LathamNevins2018}.  More generally, we may hope that the methods of \cite[\S 4]{MeyerSolleveld2012}, which give upper bounds on the radius of $\buil(G)^{T_t}$ for a maximal minisotropic torus $T$, could be adapted to the non-maximal case.  

On the other hand, we hope that this statement may be more easily proven true for the special but crucial case that $g^{-1}y$ is contained in the building $\buil(\tilde{G}^0)\subseteq \Delta$ of some complementary twisted Levi as in Definition~\ref{def:Gtilde}; by Proposition~\ref{prop:twisted-levis} these are points to which Theorem~\ref{thm:non-cusp} \emph{cannot} apply.  Recall that for such $g^{-1}y$, we have the additional hypothesis that $\tilde{Z}^0$ is a torus whose maximal parahoric fixes $g^{-1}y$, and that $\langle \tilde{Z}_0^0, Z_0^0\rangle$ generates a maximal torus $\sT$ of $\sG^0_x$.  Consequently the question could be related to one about fixed points of maximal tori.

Finally, if a point $g^{-1}y\in\Delta$ is such that the map $J\cap G_{g^{-1}y}\rightarrow J/J_+$ is surjective, then it should follow that the only way in which  $\tau(y,g)$ can fail to contain a $[G,\pi]_G$-type is in a way detectable by the simple character $\theta$---that is, if the hypotheses of Theorem \ref{thm:seville} are satisfied.

We summarize these latter two expectations in the following conjecture.

\begin{conjecture}\label{conj:seville}
Suppose that $g^{-1}y\in\Delta$ and either:
\begin{enumerate}[(i)]
\item the map $J\cap G_{g^{-1}y}\rightarrow J/J_+$ is surjective; or
\item $g^{-1}y \in \buil(\tilde{G}^0)$ for some complementary twisted Levi subgroup $\tilde{G}^0$ and %there exists a torus $\tilde{Z}^0$ such that $Z^0_0\tilde{Z}_0^0/G_{x,0+}$ is a minisotropic torus of $\sG^0_x$, and 
$s_0>1$.
\end{enumerate}
Then there exists a value $t\in(0,s_0)$ such that $[x,g^{-1}y]$ has non-empty intersection with
\[\Theta_t=\buil(G)^{H_{t+}}\backslash\buil(G)^{Z_t^0}.
\]
In particular, Theorem \ref{thm:seville} is applicable to $\tau(y,g)$, implying it does not contain a $[G,\pi]_G$-type.
\end{conjecture}

We write $\Delta_\mathrm{B}$ for the set of points in $\Delta$ which satisfy the hypotheses of Conjecture \ref{conj:seville}. Assuming this conjecture, it remains to consider those points $g^{-1}y\in\Delta\backslash\Delta_\mathrm{B}$.  As before, let $\sH$ be the image of $J_0\cap G_{g^{-1}y}$ in $\sG_x^0$.  Our expectation is that $\sH$ will be \emph{sufficiently small}, in a sense analogous to that used by Paskunas for $\mathbf{GL}_n(F)$ in \cite{Paskunas2005}.  %As such, our methods suggest that the unicity of types is inextricably linked with a (to our knowledge) open problem regarding the representation theory of finite groups of Lie type.

Before expanding on this point, it will be instructive to contrast our approach with that of \cite{Paskunas2005}. For simplicity, let us consider a datum with $s_0>1$ in order to eliminate a small number of exceptional cases. The methods of \emph{loc. cit.} may be interpreted, in the language of this paper, as partitioning $\buil(G)\backslash\buil(G)^J$ into two regions, A and B.
\begin{enumerate}[(A)]
\item For points $g^{-1}y$ in this region, $\sH$ is a proper subgroup of $\sG^0_x$.
\item For points $g^{-1}y$ in this region, $\sH=\sG_x^0$.
\end{enumerate}
For points in Region B, Paskunas makes an argument to conclude that $\tau(y,g)$ cannot contain any $[G,\pi]_G$-typical representations. This argument is a special case of our Theorem \ref{thm:seville}; the simple nature of the tori in $\mathbf{GL}_n(F)$ means that verifying the hypotheses is not difficult.  (Thus Conjecture~\ref{conj:seville} holds for $\mathbf{GL}_n(F)$.)

For points in Region A, Paskunas' approach differs significantly from ours. He shows that $\sH$ is what he calls a sufficiently small subgroup of $\sG^0_x$, which is phrased in terms of the non-maximality of the splitting fields of its elliptic elements.  He uses this to show that for most $g$, the irreducible components of $\tau(y,g)$ arise in inertially inequivalent supercuspidal representations, and hence are not $[G,\pi]_G$-types; when this is not possible, he shows that the components of $\tau(y,g)$ must intertwine with non-cuspidal representations.

%In the vast majority of cases, he is able to apply a straightforward argument based upon the Deligne--Lusztig character formula in order to show that every irreducible component of $\tau(y,g)$ must intertwine with a representation of $J$ which is of the form $\sigma'\otimes\kappa$, for some cuspidal representation $\sigma'\not\simeq\sigma$; this immediately implies that $\tau(y,g)$ cannot contain any $[G,\pi]_G$-typical representations. When this is not possible, 

In contrast, we show via Theorem \ref{thm:non-cusp} that the large majority of Mackey components from Region A (namely, those from $\buil(G)\setminus \Delta$) are contained in non-cuspidal representations.  Via Conjecture~\ref{conj:seville} we predict that Theorem~\ref{thm:seville} applies in $\Delta_B$, that is, to all components from region B, as well as those components from region A that correspond to points on the building of a complementary twisted Levi.  To address the remaining components, corresponding to points in $\Delta\setminus \Delta_B$, we make the following definition, based on the proof of \cite[Proposition 6.8]{Paskunas2005}.

\begin{definition}
Given $G$ and $x$, a subgroup $A$ of $G^0_{x,0}$, or its image $\sH$ in $\sG^0_x$, is \emph{sufficiently small} if for any cuspidal representation $\sigma$ of $\sG^0_x$, and every irreducible component $\xi$ of $\sigma|_{\sH}$, there exists a cuspidal representation $\sigma'$ of $\sG_x^0$ such that $\sigma'|_{\sH}$ contains $\xi$ and yet $\sigma'\not\simeq{}^\gamma\sigma$, for any $\gamma\in G_{x}^0$.
\end{definition}

\begin{remark}
The content of \cite[Prop. 6.8]{Paskunas2005} is that those subgroups of $\mathbf{GL}_n(\mathbb{F}_q)$ which are sufficiently small in Paskunas' sense are either also sufficiently small in the above sense, or else are  subsumed by our Theorem \ref{thm:non-cusp}.
This proof is a straightforward computation using the Deligne--Lusztig character formula, which relies heavily on special properties of the group $\mathbf{GL}_n(\mathbb{F}_q)$---specifically, it makes use of the fact that every cuspidal representation of $\mathbf{GL}_n(\mathbb{F}_q)$ is a regular Deligne--Lusztig representation, and that any minisotropic torus of $\mathbf{GL}_n(\mathbb{F}_q)$ has splitting field $\mathbb{F}_{q^n}$. Both of these assertions are often false for other finite groups of Lie type.
\end{remark}

\begin{proposition}
Suppose that $G_{g^{-1}y}\cap G^0_{x,0}$ is sufficiently small. Then $\tau(y,g)$ does not contain any $[G,\pi]_G$-typical components.
\end{proposition}

\begin{proof}
Recall from \eqref{eq:Mackey-0} that $\tau(y,g)$ is a subrepresentation of
$$
\bigoplus_{h \in (J\cap G_{g^{-1}y})\backslash J / J_0} 
\Ind_{J_0 \cap G_{g^{-1}y}}^{G_{g^{-1}y}} {}^h\lambda_0.
$$
In particular, if $\tau$ is an irreducible component of $\tau(y,g)$, then there exists an irreducible component $\xi$ of ${}^h\sigma_0|_{\mathsf{H}}$, for some $h\in J/J_0 \cong G^0_x/G^0_{x,0}$  such that $\tau$ identifies with a subrepresentation of $\Ind_{J_0\cap G_{g^{-1}y}}^{G_{g^{-1}y}}\ \xi\otimes\kappa$.

Note that the image of $G_{g^{-1}y}\cap G^0_{x,0}$ in $\sG^0_x$ coincides with $\sH$. Our assumption that $\sH$ is sufficiently small means that we may choose a cuspidal representation $\sigma'_0$ of $\sG_x^0$ such that $\sigma'_0\not\simeq{}^\gamma\sigma_0$ for any $\gamma\in G_{x}^0$, and such that $\tau$ also occurs as a subrepresentation of $\Ind_{J_0\cap G_{g^{-1}y}}^{G_{g^{-1}y}}\ \sigma'_0\otimes\kappa$. This latter representation is a (twisted) Mackey component of the representation $\Ind_{J_0}^G\ \sigma'_0\otimes\kappa$ of $G$.  By Lemma \ref{lem:S-type-connected}, it does not contain any $[G,\pi]_G$-typical representations.

%, every irreducible subquotient of which is an essentially tame supercuspidal representation.  In particular, we may choose one such subquotient $\pi'$ which contains $\tau$. Since $\pi$ is an essentially tame supercuspidal representation, it must contain an essentially tame type of the form $(J,\lambda')$, where we may take $\lambda'\simeq\sigma'\otimes\kappa$, for some irreducible representation $\sigma'$ of $J/J_+$.

%On the other hand, by assumption we have that $\tau$ is also contained in $\pi$. If $\tau$ were $[G,\pi]_G$-typical, then it would follow from Lemma \ref{lem:S-type-connected} that the restrictions to $\sG_x^0$ of $\sigma$ and $\sigma'$ intertwine with one another. By Clifford theory, each of these restrictions is a sum of conjugates of a single representation. In particular, this would imply that $\sigma_0'\simeq{}^\gamma\sigma$ for some $\gamma\in G_{x}^0$; our assumption was precisely that this is not the case.
\end{proof}

We thus anticipate that the unicity of types reduces to the following question.
  
\begin{question}
Suppose $g^{-1}y \in \Delta\backslash\Delta_\mathrm{B}$.  Is $\sH$ sufficiently small?
\end{question}

This is an aspect of an intriguing question in the representation theory of finite groups of Lie type.  Aside from some groups of type $A_n$, even in the simpler case that $G^0_x=G^0_{x,0}$, a characterization of the sufficiently small subgroups is to our knowledge an open problem, of independent interest.

\bibliographystyle{amsalpha}
\addcontentsline{toc}{chapter}{Bibliography}
\bibliography{padicrefs}

\providecommand{\bysame}{\leavevmode\hbox to3em{\hrulefill}\thinspace}
\providecommand{\MR}{\relax\ifhmode\unskip\space\fi MR }
% \MRhref is called by the amsart/book/proc definition of \MR.
\providecommand{\MRhref}[2]{%
  \href{http://www.ams.org/mathscinet-getitem?mr=#1}{#2}
}
\providecommand{\href}[2]{#2}
\begin{thebibliography}{BLR90}

\bibitem[Adl98]{Adler1998}
Jeffrey~D. Adler, \emph{Refined anisotropic {$K$}-types and supercuspidal
  representations}, Pacific J. Math. \textbf{185} (1998), no.~1, 1--32.

\bibitem[Ber84]{Bernstein1984}
Joseph Bernstein, \emph{Le ``centre'' de {B}ernstein}, Representations of
  reductive groups over a local field, Travaux en Cours, Hermann, Paris, 1984,
  Edited by P. Deligne, pp.~1--32. \MR{771671}

\bibitem[BH05]{BushnellHenniart2005}
Colin~J. Bushnell and Guy Henniart, \emph{The essentially tame local langlands
  correspondence. i}, J. Amer. Math. Soc. \textbf{18} (2005), no.~3, 685--710.
  \MR{2138141}

\bibitem[BK93]{BushnellKutzko1993}
Colin~J. Bushnell and Philip~C. Kutzko, \emph{The admissible dual of {${\rm
  GL}(N)$} via compact open subgroups}, Annals of Mathematics Studies, vol.
  129, Princeton University Press, Princeton, NJ, 1993.

\bibitem[BK98]{BushnellKutzko1998}
\bysame, \emph{Smooth representations of reductive {$p$}-adic groups: structure
  theory via types}, Proc. London Math. Soc. (3) \textbf{77} (1998), no.~3,
  582--634.

\bibitem[Blo05]{Blondel2005}
Corinne Blondel, \emph{Quelques propri\'{e}t\'{e}s des paires couvrantes},
  Math. Ann. \textbf{331} (2005), no.~2, 243--257. \MR{2115455}

\bibitem[BLR90]{BoschLutkebohmertRaynaud1990}
Siegfried Bosch, Werner L\"{u}tkebohmert, and Michel Raynaud, \emph{N\'{e}ron
  models}, Ergebnisse der Mathematik und ihrer Grenzgebiete (3) [Results in
  Mathematics and Related Areas (3)], vol.~21, Springer-Verlag, Berlin, 1990.
  \MR{1045822}

\bibitem[Bor98]{Borovoi1998}
Mikhail Borovoi, \emph{Abelian {G}alois cohomology of reductive groups}, Mem.
  Amer. Math. Soc. \textbf{132} (1998), no.~626, viii+50. \MR{1401491}

\bibitem[BT84]{BruhatTits1984}
Fran\c{c}ois Bruhat and Jacques Tits, \emph{Groupes r\'eductifs sur un corps
  local. {II}. {S}ch\'emas en groupes. {E}xistence d'une donn\'ee radicielle
  valu\'ee}, Inst. Hautes \'Etudes Sci. Publ. Math. (1984), no.~60, 197--376.

\bibitem[Cas73]{Casselman1973}
William Casselman, \emph{The restriction of a representation of {${\rm
  GL}_{2}(k)$} to {${\rm GL}_{2}({\mathfrak{o}})$}}, Math. Ann. \textbf{206}
  (1973), 311--318.

\bibitem[CN10]{CampbellNevins2010}
Peter~S. Campbell and Monica Nevins, \emph{Branching rules for ramified
  principal series representations of {GL}(3) over a {$p$}-adic field}, Canad.
  J. Math. \textbf{62} (2010), no.~1, 34--51.

\bibitem[Fin15]{Fintzen2015}
Jessica Fintzen, \emph{On the {M}oy-{P}rasad filtration}, Preprint
  \texttt{arxiv.org} arXiv:1511.00726v3 [math.RT], 2015.

\bibitem[Fin18]{Fintzen2018}
\bysame, \emph{Types for tame $p$-adic groups}, Preprint \texttt{arxiv.org}
  arXiv:1810.04198 [math.RT], 2018.

\bibitem[Hak18]{Hakim2018}
Jeffrey Hakim, \emph{Constructing tame supercuspidal representations},
  Represent. Theory \textbf{22} (2018), 45--86. \MR{3817964}

\bibitem[Han87]{Hansen1987}
Kristina Hansen, \emph{Restriction to {${\rm GL}_2({\mathcal{O}})$} of
  supercuspidal representations of {${\rm GL}_2(F)$}}, Pacific J. Math.
  \textbf{130} (1987), no.~2, 327--349.

\bibitem[HM08]{HakimMurnaghan2008}
Jeffrey Hakim and Fiona Murnaghan, \emph{Distinguished tame supercuspidal
  representations}, Int. Math. Res. Pap. IMRP (2008), no.~2, Art. ID rpn005,
  166.

\bibitem[Kim07]{Kim2007}
Ju-Lee Kim, \emph{Supercuspidal representations: an exhaustion theorem}, J.
  Amer. Math. Soc. \textbf{20} (2007), no.~2, 273--320 (electronic).

\bibitem[Kot97]{Kottwitz1997}
Robert~E. Kottwitz, \emph{Isocrystals with additional structure. {II}},
  Compositio Math. \textbf{109} (1997), no.~3, 255--339. \MR{1485921}

\bibitem[Kut77]{Kutzko1977}
P.~C. Kutzko, \emph{Mackey's theorem for nonunitary representations}, Proc.
  Amer. Math. Soc. \textbf{64} (1977), no.~1, 173--175.

\bibitem[KY17]{KimYu2017}
Ju-Lee Kim and Jiu-Kang Yu, \emph{Construction of tame types}, Representation
  theory, number theory, and invariant theory, Progr. Math., vol. 323,
  Birkh\"{a}user/Springer, Cham, 2017, pp.~337--357. \MR{3753917}

\bibitem[Lat16]{Latham2016}
Peter Latham, \emph{Unicity of types for supercuspidal representations of
  {$p$}-adic {$\mathrm{SL}_2$}}, J. Number Theory \textbf{162} (2016),
  376--390. \MR{3448273}

\bibitem[Lat17]{Latham2017}
\bysame, \emph{The unicity of types for depth-zero supercuspidal
  representations}, Represent. Theory \textbf{21} (2017), 590--610.
  \MR{3735454}

\bibitem[Lat18]{Latham2018}
\bysame, \emph{On the unicity of types in special linear groups}, Manuscripta
  Math. \textbf{157} (2018), no.~3-4, 445--465. \MR{3858412}

\bibitem[LN19]{LathamNevins2018}
Peter Latham and Monica Nevins, \emph{On the unicity of types for toral
  supercuspidal representations}, Representations of reductive {$p$}-adic
  groups, Progr. Math., vol. 328, Birkh\"{a}user/Springer, Singapore, 2019,
  pp.~175--190. \MR{3889763}

\bibitem[Mor99]{Morris1999}
Lawrence Morris, \emph{Level zero {$\bf G$}-types}, Compositio Math.
  \textbf{118} (1999), no.~2, 135--157.

\bibitem[MP96]{MoyPrasad1996}
Allen Moy and Gopal Prasad, \emph{Jacquet functors and unrefined minimal
  {$K$}-types}, Comment. Math. Helv. \textbf{71} (1996), no.~1, 98--121.

\bibitem[MS12]{MeyerSolleveld2012}
Ralf Meyer and Maarten Solleveld, \emph{Characters and growth of admissible
  representations of reductive {$p$}-adic groups}, J. Inst. Math. Jussieu
  \textbf{11} (2012), no.~2, 289--331. \MR{2905306}

\bibitem[Nad17]{Nadimpalli2017}
Santosh Nadimpalli, \emph{Typical representations for level zero {B}ernstein
  components of {${\rm GL}_n(F)$}}, J. Algebra \textbf{469} (2017), 1--29.
  \MR{3563005}

\bibitem[Nad19]{Nadimpalli2019}
\bysame, \emph{On classification of typical representations for {${\rm
  GL}_3(F)$}}, Forum Math. \textbf{31} (2019), no.~4, 917--941. \MR{3975668}

\bibitem[Nev05]{Nevins2005}
Monica Nevins, \emph{Branching rules for principal series representations of
  {${\rm SL}(2)$} over a {$p$}-adic field}, Canad. J. Math. \textbf{57} (2005),
  no.~3, 648--672.

\bibitem[Nev13]{Nevins2013}
\bysame, \emph{Branching rules for supercuspidal representations of
  {$SL_2(k)$}, for {$k$} a {$p$}-adic field}, J. Algebra \textbf{377} (2013),
  204--231.

\bibitem[Nev14]{Nevins2014}
\bysame, \emph{On branching rules of depth-zero representations}, J. Algebra
  \textbf{408} (2014), 1--27. \MR{3197168}

\bibitem[OS14]{OnnSingla2014}
Uri Onn and Pooja Singla, \emph{On the unramified principal series of
  {$\mathrm{GL}(3)$} over non-{A}rchimedean local fields}, J. Algebra
  \textbf{397} (2014), 1--17.

\bibitem[Pas05]{Paskunas2005}
Vytautas Paskunas, \emph{Unicity of types for supercuspidal representations of
  {${\rm GL}_N$}}, Proc. London Math. Soc. (3) \textbf{91} (2005), no.~3,
  623--654. \MR{2180458}

\bibitem[PR08]{PappasRappoport2008}
G.~Pappas and M.~Rapoport, \emph{Twisted loop groups and their affine flag
  varieties}, Adv. Math. \textbf{219} (2008), no.~1, 118--198, With an appendix
  by T. Haines and Rapoport. \MR{2435422}

\bibitem[SS08]{SecherreStevens2008}
V.~S\'{e}cherre and S.~Stevens, \emph{Repr\'{e}sentations lisses de {${\rm
  GL}_m(D)$}. {IV}. {R}epr\'{e}sentations supercuspidales}, J. Inst. Math.
  Jussieu \textbf{7} (2008), no.~3, 527--574. \MR{2427423}

\bibitem[Ste05]{Stevens2005}
Shaun Stevens, \emph{Semisimple characters for {$p$}-adic classical groups},
  Duke Math. J. \textbf{127} (2005), no.~1, 123--173. \MR{2126498}

\bibitem[Ste08]{Stevens2008}
\bysame, \emph{The supercuspidal representations of {$p$}-adic classical
  groups}, Invent. Math. \textbf{172} (2008), no.~2, 289--352.

\bibitem[Yu01]{Yu2001}
Jiu-Kang Yu, \emph{Construction of tame supercuspidal representations}, J.
  Amer. Math. Soc. \textbf{14} (2001), no.~3, 579--622 (electronic).

\bibitem[Yu15]{Yu2015}
\bysame, \emph{Smooth models associated to concave functions in {B}ruhat-{T}its
  theory}, Autour des sch\'{e}mas en groupes. {V}ol. {III}, Panor. Synth\`eses,
  vol.~47, Soc. Math. France, Paris, 2015, pp.~227--258. \MR{3525846}

\end{thebibliography}

\end{document}